\documentclass[oneside, 11pt]{amsart}
\title[] {On Newton-Okounkov bodies of Graded Linear Series}
\author{Georg Merz} 
\usepackage{xcolor}

\usepackage{enumerate}
\usepackage{amsmath}
\usepackage{mathrsfs}
\usepackage{amssymb}
\usepackage{amsthm}
\usepackage[all]{xy}
\usepackage{tikz}
\usepackage{mathtools}  
\mathtoolsset{showonlyrefs}
\usepackage{bookmark}

\date{\today}

\address{Georg Merz,
Mathematisches Institut,
Georg-August-Universit\"at G\"ottingen,
Bunsenstra\ss e 3-5, 
D-37073 G\"ottingen,
Germany}
\email{georg.merz@mathematik.uni-goettingen.de}

\DeclareMathOperator{\Sb}{S_\bullet}
\DeclareMathOperator{\spec}{Spec}
\DeclareMathOperator{\proj}{Proj}
\DeclareMathOperator{\ind}{ind}

\DeclareMathOperator{\vol}{vol}

\newcommand{\f}{\mathcal{F}}

\newcommand{\bs}{\mathbf{B}}
\newcommand{\bsp}{\mathbf{B_{+}}}

\newcommand{\bsi}{\mathfrak{b}}

\DeclareMathOperator{\conv}{conv}
\DeclareMathOperator{\fract}{Quot}

\DeclareMathOperator{\codim}{codim}
\DeclareMathOperator{\bl}{Bs}

\newcommand{\p}{\mathfrak{p}}

\newcommand{\struc}{\mathcal{O}}

  \theoremstyle{plain}
  \newtheorem{thm}{Theorem}[section]
  \newtheorem*{quest}{Question}
  \newtheorem{lem}[thm]{Lemma}
  \newtheorem{cor}[thm]{Corollary}
  \newtheorem{prop}[thm]{Proposition}
  \theoremstyle{definition}
  \newtheorem{defi}[thm]{Definition}
  \newtheorem{ex}[thm] {Example}
\theoremstyle{theorem}
\newtheorem{rem}[thm]{Remark}

\newtheorem{thmx}{Theorem}

\begin{document}
\begin{abstract}

We  generalize the theory of Newton-Okounkov bodies of big divisors to the case of graded linear series. One of the results is the generalization of slice formulas and the existence of generic Newton-Okounkov bodies for birational graded linear series. We  also give a characterization of graded linear series which have full volume in terms of their base locus.
\end{abstract}
\maketitle

\section{Introduction}
The beautiful paper \cite{Ok96}  by Andrei Okounkov began the theory of Newton-Okounkov bodies. Originally, he was interested in asymptotic multiplicities of group representations on line bundles. To study this, he constructed a compact convex set $\Delta$ and noticed that the volume of this body can be interpreted as the asymptotic multiplicity of the given representation.  
Inspired by that work,  Kaveh and Khovanskii in \cite{KK12} and independently Lazarsfeld and Musta{\c{t}}{\u{a}} in \cite{LM09} realized that Okounkov's construction gave rise to a completely new approach for studying the asymptotics of linear series on a projective variety. They  used Okounkov's construction to associate a convex body $\Delta(S_\bullet)$ to a graded linear series $S_\bullet$ on a projective variety $X$.  
The construction of this body does not only depend on $S_\bullet$, but also on the choice  of a flag $Y_\bullet$ consisting of closed irreducible subvarieties.
 However, the main feature of the convex body $\Delta_{Y_\bullet}(S_\bullet)$ is that, under ``mild'' conditions, its euclidean volume gives a geometric interpretation of the classical notion of the volume of a graded linear series $S_\bullet$ \cite[Theorem A]{LM09}. Indeed, for a graded linear series $S_\bullet$ on a variety $X$ of dimension $d$ one has
\begin{align}
\vol ( \Delta_{Y_\bullet}(S_\bullet))= \lim_{k\to \infty} \frac{\dim S_k}{k^n}= d!\cdot \vol(S_\bullet).
\end{align}
 A posteriori, one concludes that the volume of $\Delta_{Y_\bullet}(S_\bullet)$ is independent of the choice of the flag $Y_\bullet$.   \\

 For a complete graded linear series $S_\bullet$   equal to the full section algebra  \[R_{\bullet}(X,D):=\bigoplus_{k\in\mathbb{N}} H^0(X,\mathcal{O}_X(kD))\] associated to a big divisor $D$, we get a very nice correspondence between algebraic geometry and convex geometry.
 Although,  a priori, the Newton-Okounkov body of a divisor $D$  is well-defined up to linear equivalence of $D$, Lazarsfeld and Musta{\c{t}}{\u{a}} \cite[Proposition 4.1]{LM09} showed that the construction of the Newton-Okounkov body does only depend on the numerical equivalence class. Conversely, Jow showed in \cite{jow}  that two divisors $D$ and $D^\prime$ are numerical equivalent if the Newton-Okounkov bodies $\Delta_{Y_\bullet}(D)$ and $\Delta_{Y_\bullet}(D^\prime)$ coincide for all  flags $Y_\bullet$. Philosophically, this means that we can interpret a numerical equivalence class  of a divisor $D$ as a collection of real convex bodies parametrized by the set of all flags. Therefore, in principle, it should be possible to translate all numerical properties of a divisor into properties of convex geometry and vice versa. An outline of this approach is given in recent works of K\"uronya and Lozovanu (see \cite{KL14},\cite{KL15} and \cite{KurLoz15}).
 
Even though Lazarsfeld and Musta{\c{t}}{\u{a}} tried to build up their theory of Newton-Okounkov bodies in a general setting based on graded linear series, many statements  are only formulated for complete graded linear series corresponding to a big divisor $D$, i.e. for $R(S_\bullet)=R_{\bullet}(X,D)$.
The reason why this case is significantly easier to understand is due to the following facts which are not shared by arbitrary graded linear series.
\begin{itemize} 
\item The algebra $R_{\bullet}(X,D)$ is induced by the locally free sheaves $\struc_X(kD)$.
\item The body $\Delta_{Y_\bullet}(D)$ is well defined for a numerical equivalence class of $D$ which can be interpreted as a point in a finite dimensional vector space over $\mathbb{R}$, the N\'eron-Severi space.
\item There exists a global Newton-Okounkov body which characterizes all Newton-Okounkov bodies $\Delta_{Y_\bullet}(D)$ at once \cite[Theorem B]{LM09}.
\end{itemize}
 The two main features of Newton-Okounkov bodies which were proved  in the case $S_\bullet=R_{\bullet}(X,D)$ but were left open for arbitrary graded linear series are the following .
 \begin{enumerate}
 \item \textit{Slice formulas for Newton-Okounkov bodies}: Let $t\geq 0$ be a rational number. Let $Y_\bullet$ be a flag such that $Y_1$ is a Cartier divisor. Suppose $D$ is a big Cartier divisor such that $Y_1\not\subseteq \bsp(D)$ and $D-t\cdot Y_1$ is big. Then the $t$-slice  \[\Delta_{Y_\bullet}(D)_{\nu_1=t}:=\Delta_{Y_\bullet}(D)\cap \left(\{t\}\times \mathbb{R}^{d-1}\right) \] is equal to the Newton-Okounkov body of the restricted linear series \cite[Theorem 4.24]{LM09} \[R_{\bullet}(X,D-tY_1)_{|Y_1}.\] 
 \item\label{b} \textit{Existence of a generic Newton-Okounkov body}:
 If we have a family of Newton-Okounkov bodies $\Delta_{Y_{t,\bullet}}(X_t,D_t)$ where all the relevant data move in  flat families over $T$, then for a very general choice of $t\in T$ the Newton-Okounkov bodies all coincide \cite[Theorem 5.1]{LM09}. 
 \end{enumerate}
 \begin{quest}
 Is there a natural generalization of the above properties for more general graded linear series $S_\bullet$?
 \end{quest}
 We will prove that property $(a)$ for rational $t>0$ does  hold for a completely  arbitrary graded linear series $S_\bullet$ corresponding to a big divisor.
 \begin{thmx}
 Let $S_\bullet$ be a graded linear series. Let $Y_\bullet$ be an admissible flag and $t=a/b>0$ for $(a,b)=1$ a rational number such that $\{t\}\times \mathbb{R}^{d-1}$ meets the interior of $\Delta_{Y_\bullet}(S_\bullet)$. Then 
\begin{align}
\Delta_{Y_\bullet}(S_\bullet)_{\nu_1=t}=1/b\cdot \Delta_{X|Y_1}(S^{(b)}_\bullet -a Y_1)
\end{align}
via the identification of $\{t\}\times \mathbb{R}^{d-1}\cong \mathbb{R}^{d-1}$.
 \end{thmx} 
 However, for $t=0$ we need some more constraints on  $S_\bullet$. It turns out that a natural assumption in order to treat $S_\bullet$ 
``like'' a divisor $D$ is to assume that their volumes are equal, i.e. $\vol(S_\bullet)=\vol(D)$. Another assumption which is necessary to prevent $S_\bullet$ from being too wild is that it should be finitely generated as an algebra. 
 We will prove the following characterization of such graded linear series.
\begin{thmx}
Let $S_\bullet\subseteq R(X,D)$ be  finitely generated graded linear series corresponding to a big divisor $D$.  Then the following two conditions are equivalent
\begin{enumerate}
\item $\vol(S_\bullet)=\vol(D)$.
\item 
\begin{itemize}
\item The rational map $h_{S_\bullet}\colon X \dashrightarrow \proj (S_\bullet) $ is birational and
\item $\bs(S_\bullet)=\emptyset $ on $\proj(R(X,D))$.
\end{itemize}
\end{enumerate}
\end{thmx}
The above result enables us  to derive a slice formula for $t=0$ for such graded linear series (see Theorem \ref{thmsliceful}). However, we are even able to derive the following slice theorem for graded linear series containing an ample series.
\begin{thmx}
Let $S_\bullet$ be a graded linear series containing the ample series $D-E$. Let $Y_\bullet$ be an admissible flag such that the divisorial component $Y_1$ is not contained in $E$ and $Y_d\not \in \bs(S_\bullet)$. Then we have
\begin{align}
\Delta_{Y_\bullet}(S_\bullet)_{\nu_1=0}=\Delta_{X|Y_1}(S_\bullet).
\end{align}
\end{thmx}

The existence of a generic Newton-Okounkov body, according to Part (b), can be generalized to the case of  birational graded linear series $S_\bullet$. More precisely, we derive the following theorem.
\begin{thmx}
Let $X, T$ and $\mathcal{Y}_\bullet$ be as in Section \ref{famok}. Let $S_\bullet$ be a birational  graded linear series.
Then for a very general choice of $t\in T$ all the Newton-Okounkov bodies
$\Delta_{Y_{t,\bullet}}(S_\bullet)$
coincide. 
\end{thmx}
Note that we do not need any flatness hypothesis for the above theorem. The idea of the proof is that we replace the graded linear series $S_\bullet$ by a possibly larger one which is induced by  coherent sheaves and then use the generic flatness theorem to make sure all the involved data is flat. In addition to proving the existence of generic Newton-Okounkov bodies,
we  give  several examples of how to construct such families.

\section*{Acknowledgement}
The author would like to thank Harold Blum and Marcel Maslovari\'c for helpful discussions.
The author owes his deepest gratitude to his supervisor Henrik Sepp\"anen for his patient guidance, enthusiastic encouragement and useful critiques of this research work.  
 
\section{Preliminaries}
\subsection{Notation}
We work over the field $\mathbb{C}$. Whenever not otherwise stated, $X$ denotes a projective variety over $\mathbb{C}$ of dimension $d$ and $D$ a big Cartier divisor over $X$. Moreover, when we talk of a divisor, we  will always refer to an  integral Cartier divisor.
By an admissible flag $Y_\bullet$ of $X$ we mean an ordered set of irreducible subvarieties:
\begin{align}
Y_\bullet \colon X=Y_0\supseteq Y_1 \supseteq Y_2\supseteq \dots \supseteq Y_{d-1}\supseteq Y_{d}=\{pt\}
\end{align}
where $d$ is the dimension of $X$ such that $\codim_{X}(Y_i)=i$ and each $Y_i$ is non-singular at the point $Y_d$.

\subsection{Graded linear series}
Let $D$ be a divisor on $X$ and let $V\subseteq H^0(X,\struc_X(D))$ be a non-zero vector subspace. Then the projective space of one dimensional  vector subspaces of $V$, which we denote by $|V|:=\mathbb{P}(V)$, is called a linear series. In case $V=H^0(X,\struc_X(D))$ we call it the complete linear series and write $|D|$. Often we are interested in the asymptotic behaviour of $|mD|$ as $m\to \infty$. In order to generalize this for non complete linear series we need the following definition.
\begin{defi}
A \emph{graded linear series} on $X$ corresponding to a divisor $D$ consists of a collection
\begin{align}
S_\bullet = \{ S_k \}_{k\geq 0}
\end{align}
of finite dimensional vector subspaces $S_m\subseteq H^0(X,\struc_X(mD))$ and $S_0=\mathbb{C}$. These subspaces are required to satisfy the property 
\begin{align}
S_k\cdot S_l \subseteq S_{k+l} \quad \text{for all } k,l  \geq 0
\end{align}
where $S_k \cdot S_l$ denotes the image of $S_k\otimes S_l$ under the homomorphism
\begin{align}
H^0(X,\struc_X(kD)) \otimes H^0(X,\struc_X(lD)) \to H^0(X,\struc_X ((k+l)D)).
\end{align}
We call $S_\bullet$ a complete linear series if $S_k=H^0(X,\struc_X(kD))$ for all $k\geq 0$.
\end{defi}
Given a graded linear series $S_\bullet$ we can define the \emph{graded algebra}
\begin{align}
R(S_\bullet):= \bigoplus^\infty_{k=0} S_k.
\end{align}
If $S_\bullet$ is a complete graded linear series corresponding to $D$, we write $R_{\bullet}(X,D)$ for the graded algebra of sections.
We say that $S_\bullet$ is \emph{finitely generated} if $R(S_\bullet)$ is finitely generated as a $\mathbb{C}$-algebra.

A linear series $|V|$ corresponds to a rational morphism 
\begin{align}
h_{|V|}\colon X \dashrightarrow \mathbb{P}^N.
\end{align}
\begin{defi}
We say that a graded linear series $S_\bullet$ is \emph{birational} if the rational map corresponding to the linear series $|S_k|$
\begin{align}
h_{|S_k|}\colon X \dashrightarrow \mathbb{P}^N 
\end{align}
is birational onto its image for $k\gg 0$.
\end{defi}
If $S_\bullet$ is finitely generated, then $S_\bullet$ is birational if and only if the induced rational map
\begin{align}
h_{S_\bullet} \colon X \dashrightarrow \proj (S_\bullet):=\proj(R(S_\bullet))
\end{align}
is birational.

If $S_\bullet$ and $T_\bullet$ are two graded linear series, we write $S_\bullet\subseteq T_\bullet$ if $S_k\subseteq T_k$ for all $k\geq 0$.

Finally, note that many  notions as the volume or the stable base locus can be defined for graded linear series completely analogously as in the case of complete graded linear series (see \cite[2.4]{laz} for more details).

\subsection{Construction of  Newton-Okounkov bodies}
In this section we want to give  a very brief overview of the construction of  Netwon-Okounkov bodies and state some elementary facts about them. For a detailed overview see e.g. \cite{LM09}.

First of all we fix an admissible flag $Y_\bullet$ and a graded linear series $S_\bullet$ of $X$. Now by an iterative procedure, taking the order of vanishing along the given $Y_i$ into account, we construct for each $k\in\mathbb{N}$ a valuation map
\begin{align}
\nu_{Y_\bullet} \colon S_k \setminus \{0 \} \to \mathbb{Z}^d.
\end{align}
The two essential properties of $\nu_{Y_\bullet}$ are:
\begin{itemize}

\item ordering $\mathbb{Z}^d$ lexicographically, we have
\begin{align}
\nu_{Y_\bullet}(s_1+s_2)\geq \min \{ \nu_{Y_\bullet}(s_1),\nu_{Y_\bullet}(s_2) \}
\end{align}
for any $s_1,s_2 \in S_k\setminus \{0\}$
\item given two non zero sections $s\in S_k$ and $t\in S_l$ then
\begin{align}
\nu_{Y_\bullet}(s\otimes t )= \nu_{Y_\bullet}(s)+ \nu_{Y_\bullet}(t).
\end{align}
\end{itemize}
The valuation function gives rise to the semigroup
\begin{align}
\Gamma(S_\bullet):= \{ (\nu_{Y_\bullet}(s),k) \ : \ s\in S_k\setminus\{0\}, k\in \mathbb{N} \} \subseteq \mathbb{N}^{d+1}.
\end{align}
Then the Newton-Okounkov body of $S_\bullet$ corresponding to the flag $Y_\bullet$ is given by
\begin{align}
\Delta_{Y_\bullet}(S_\bullet):= \overline {\text{Cone}(\Gamma(S_\bullet))}\cap \left( \mathbb{R}^d\times \{1\}\right).
\end{align}
In \cite{LM09} it was shown that for a graded linear series $S_\bullet$ corresponding to a big divisor $D$ which has the additional property that the semigroup $\Gamma(S_\bullet)$ generates $\mathbb{Z}^{d+1}$ as a group, we have
\begin{align}
\vol_{\mathbb{R}^d} (\Delta_{Y_\bullet}(S_\bullet))=\frac{1}{d!} \cdot \vol (S_\bullet)
\end{align}
where 
\begin{align}
\vol(S_\bullet):= \lim_{k\to \infty} \frac{\dim S_k}{k^d/d!}.
\end{align}
%It was also shown that for a birational graded linear series, the semigroup $\Gamma(S_\bullet)$ generates $\mathbb{Z}^{d+1}$ as a group for all admissible flags $Y_\bullet$. A posteriori, we get that for such graded linear series the volume of the Newton-Okounkov body does not depend on the choice of the flag.
However, the more general case was treated in \cite{KK12}. They showed that for an arbitrary graded linear series $S_\bullet$, we have
\begin{align}\label{volkk}
\frac{\vol(\Delta_{Y_\bullet}(S_\bullet))}{\ind(S_\bullet)}= \frac{ \vol (S_\bullet)} {d!}
\end{align}
where $\ind (S_\bullet)$ is the index of the group generated by $\Gamma(S_\bullet)$ in $\mathbb{Z}^d$.
 
 So for arbitrary graded linear series the volume of the Newton-Okounkov body does indeed depend on the choice of the flag.

\section{Volume and base Locus of graded linear series}
In this section we want to analyze the correspondence between the volume of a  graded linear series and its base locus. We first focus on the case where the graded linear series $S_\bullet$ corresponding to $D$ has full volume, i.e.
\begin{align}
\vol(S_\bullet)=\vol(D).
\end{align}
In this case we have a characterization of finitely generated graded linear series $S_\bullet$ given by Theorem \ref{thmcharempty}. This characterization will help us to make sense of the sheafication of a graded linear series, which will be necessary for Section \ref{secgen}, as well as for deriving slice formulas in the following section.

\subsection{Stable base locus and volume of finitely generated graded linear series}
%maybe enter some definition of stable base locus in her

The aim of this paragraph is to to show that two finitely generated graded linear series which have the same volume also have the same stable base locus.

The following  proposition will be helpful.
\begin{prop}
Let $S_\bullet$ be a graded linear series and $x\in X$. Consider the induced graded linear series $W_\bullet$ defined by
 \[W_k:=\{ s\in S_k \ : \ \operatorname{ord}_x(s)\geq \lceil kr \rceil \} \] for some fixed $r>0$. Then for all admissible flags $Y_\bullet$ centered at $x$, the origin $0$ does not lie in the Newton-Okounkov body $\Delta_{Y_\bullet}(W_\bullet)$.
\end{prop}
\begin{proof}
Let us assume that $0$ lies in the Newton-Okounkov body. Then it must be an extreme point. Thus there exists  a series of sections $s^k\in W_k$ such that $1/k\cdot\nu(s^k)=1/k \cdot (\nu_1(s^k),\dots \nu_d(s^k))$ converges to $0$ as $k$ tends to infinity. By \cite[Lemma 2.4]{KL15}, we have
\begin{align}
\operatorname{ord}_x(s^k)\leq \sum_{i=1}^d \nu_i(s^k).
\end{align}
Dividing by $k$ leads to
\begin{align}
1/k \cdot \operatorname{ord}_x(s^k)\leq 1/k \cdot \sum_{i=1}^d \nu_i(s^k).
\end{align}
As $k$ tends to infinity the right hand side goes to $0$, but the left hand side is lower bounded by $r>0$, which gives a contradiction.
Thus $0$ does not lie in the Newton-Okounkov body $\Delta_{Y_\bullet}(W_\bullet)$.
\end{proof}

The next theorem is an analog of \cite{KL15} Theorem A for finitely generated graded linear series.
\begin{thm}\label{gradedfiniteKL} Let $X$ be smooth and
 let $S_\bullet$ be a finitely generated graded linear series. Then the following  conditions are equivalent:
\begin{enumerate}
\item $x\not\in \bs (S_\bullet)$
\item $0\in \Delta_{Y_\bullet}(S_\bullet)$ for each admissible flag $Y_\bullet$ centered  at $\{x\}$
\item There exists an admissible flag $Y_\bullet$ centered at $x$ such that $0\in\Delta_{Y_\bullet}(S_\bullet).$
\end{enumerate}
\end{thm}
\begin{proof}
$(a) \rightarrow (b)$ is trivial, since $\nu_{Y_\bullet}(s)=0$ for all sections $s\in S_k$ that do not vanish at $x$.
$(b)\rightarrow (c)$ is also trivial.

Let us prove $(c)\rightarrow (a)$. Let $0$ lie in the Newton-Okounkov body $\Delta_{Y_\bullet}(S_\bullet)$ and let us assume that $x\in \bs(S_\bullet)$. 
Let $s_1,\dots, s_n$ be homogeneous generators of $R(S_\bullet)$ and denote by $N$ the maximum of the degrees of these generators.
 By assumption, all these generators vanish at $x$ to order at least one. Thus, we have an inclusion $S_\bullet\subseteq W_\bullet$ of graded linear series, where 
\begin{align} 
 W_k:=\{s\in H^0(X,\mathcal{O}_X(kD)) 
\ : \ \operatorname{ord}_x(s)\geq \lceil k/N\rceil \}.
\end{align}
  But by the previous proposition, $0\not \in \Delta(W_k)$. This contradicts the fact that $0\in \Delta_{Y_\bullet}(S_\bullet)\subseteq \Delta_{Y_\bullet}(W_\bullet)$. Thus $x\not \in \bs(S_\bullet)$.
\end{proof}
\begin{lem} \label{lemvolok}
Let $S_\bullet\subseteq T_\bullet$ be two graded linear series. Then $\vol(S_\bullet)=\vol(T_\bullet)\neq 0$ implies that $\Delta_{Y_\bullet}(S)=\Delta_{Y_\bullet}(T_\bullet)$ for all admissible flags $Y_\bullet$.
\end{lem}
\begin{proof}
First of all, we show that for all admissible flags the volume of the Newton-Okounkov bodies coincide. We will do this by showing that the indices of the semigroups $\Gamma(S_\bullet)$ and $\Gamma(T_\bullet)$ are equal. Clearly, $\Gamma(S_\bullet)\subseteq \Gamma(T_\bullet)$ and hence $\ind(S_\bullet)\geq \ind(T_\bullet)$. On the other hand, the volume formula for Newton-Okounkov bodies yields  the equality
\begin{align}
d! \cdot \vol(S_\bullet)=\frac{\vol(\Delta_{Y_\bullet}(S_\bullet))}{\ind(S_\bullet)}= 
\frac{\vol(\Delta_{Y_\bullet}(T_\bullet))}{\ind(T_\bullet)}=d! \cdot \vol(T_\bullet).
\end{align}
From this equality and the fact that $\vol(\Delta_{Y_\bullet}(S_\bullet))\leq \vol (\Delta_{Y_\bullet}(T_\bullet))$ we get that $\ind(S_\bullet)\leq \ind(T_\bullet)$, which implies $\ind(S_\bullet)=\ind(T_\bullet)$. Again from the volume formula, we deduce that the volume of the Newton-Okounkov bodies are equal.
Let us now assume that there is a flag $Y_\bullet$ such that $\Delta_{Y_\bullet}(S_\bullet)\subsetneq \Delta_{Y_\bullet}(T_\bullet)$. Then there is a point $P$ in $\Delta_{Y_\bullet}(T_\bullet)$ which does not lie in $\Delta_{Y_\bullet}(S_\bullet)$.
Since $\Delta_{Y_\bullet}(S_\bullet)$ is closed and convex, the point $P$ has a positive distance to $\Delta_{Y_\bullet}(S_\bullet)$. Hence, there is a  $d$-dimensional ball $B(0,\varepsilon)$ around the origin which does not intersect $\Delta_{Y_\bullet}(S_\bullet)$. The intersection of $B(0,\varepsilon)$ with $\Delta_{Y_\bullet}(T_\bullet)$ has positive volume. This shows that we cannot have $\vol(\Delta_{Y_\bullet}(S_\bullet))=\vol(\Delta_{Y_\bullet}(T_\bullet))$.
\end{proof}
For the next lemma we will need the definition of a pulled back linear series.
Let $\pi\colon X\to Y$ be a morphism of projective varieties and $S_\bullet$ a graded linear series on $Y$. Then we can define $\pi^*S_\bullet$ by $\pi^*S_k:=\{\pi^*s \ : \ s\in S_k\}$.

\begin{lem} \label{lemnotsmooth}
Let $\pi\colon X\to Y$ be a  surjective morphism of projective varieties. Let $S_\bullet$ be a graded linear series on $Y$. Then we have 
\begin{align}
 \pi(\bs(\pi^*S_\bullet)= \bs(S_\bullet)
 \end{align}
 \end{lem}
\begin{proof}
Let $x\in \bs(\pi^*S_\bullet)$ this is equivalent to 
\begin{align}
\pi^*s(x)=s(\pi(x))=0
\end{align} for all $k\geq 0$ and $s\in S_k$.
But this is equivalent to $\pi(x)\in \bs(S_\bullet)$.
Hence, we have the desired result.

\end{proof}

\begin{thm} \label{thmvolbase2}
Let $S_\bullet \subseteq T_\bullet$ be  two graded linear series and let $S_\bullet$ be finitely generated. Then $\vol(S_\bullet)=\vol(T_\bullet)$ implies that $\bs(S_\bullet)=\bs(T_\bullet)$.
\end{thm}
\begin{proof}
%We first prove the assertion for $X$ being smooth.

Let us first assume that $X$ is smooth.
It is obvious that $\bs(T_\bullet)\subseteq \bs(S_\bullet)$. Let us show the other inclusion. Let therefore $x\in \bs(S_\bullet)$ and  assume that $x$ does not lie in $\bs(T_\bullet)$. Then for all admissible flags $Y_\bullet$ centered at $\{x\}$ we have that $0\in \Delta_{Y_\bullet} (T_\bullet)$. By Theorem \ref{gradedfiniteKL}, we know that $0\not \in \Delta_{Y_\bullet}(S_\bullet)$. Using Lemma \ref{lemvolok}  implies that $\Delta_{Y_\bullet}(S_\bullet)=\Delta_{Y_\bullet}(T_\bullet)$, which gives a contradiction. Thus $x$ does lie in $\bs(T_\bullet)$.

Now consider the case where $X$ is not necessarily smooth and $\pi\colon \tilde{X}\to X$ is a resolution of singularities. Since we have bijection of sections $\pi^*S_k\cong S_k$ and $\pi^*T_k\cong T_k$ we conclude $\vol(\pi^*T_\bullet)=\vol(\pi^*S_\bullet)$. Since $\tilde{X}$ is smooth we can deduce that $\bs(\pi^*S_\bullet)=\bs(\pi^*T_\bullet)$ and from the above lemma the desired result follows.
\end{proof}
The next example illustrates that the assumption for $S_\bullet$ to be finitely generated is indeed necessary.
\begin{ex}
Let $S_\bullet$ be an arbitrary graded linear series such that $\vol(S_\bullet)>0$. Then choose any point $x\in X\setminus \bs(S_\bullet)$ and consider the graded linear series $S^x_\bullet$ defined by
\begin{align}
S_k^x:=\{s\in S_k \ :\ s(x)=0\}.
\end{align}
Clearly, we have $\bs(S_\bullet)\neq \bs(S^x_\bullet)$ since $x$ is not contained in the first set but is contained in the latter by construction. We will nevertheless show that $\vol(S_\bullet)=\vol(S^x_\bullet)$,  and can therefore conclude that $S^x_\bullet$ is never finitely generated, even though $S_\bullet$ might be.

In order to prove the equality of volumes, we first show equality of Newton-Okounkov bodies. Let $Y_\bullet$ be any admissible flag.
We surely have an inclusion $\Delta_{Y_\bullet}(S^x_\bullet)\subseteq \Delta_{Y_\bullet}(S_\bullet)$.
Let now $P$ be a point in $\Delta_{Y_\bullet}(S_\bullet)$. Using the fact that the valuation points in the Newton-Okounkov body are dense (see Theorem \ref{thmvalrat}), there is a series of sections $(\xi^{m_i})_{i\in \mathbb{N}}$ such that $\xi^{m_i}\in S_{m_i}$ and
\begin{align}
\frac{\nu_{Y_\bullet}(\xi^{m_i})}{m_i}  \to P \quad \text{ as } m_i  \to \infty.
\end{align}

Now let $s\in S_k$ be any section such that $s(x)=0$. The existence of such a section follows from the fact that $\vol(S_\bullet)>0$.
Suppose otherwise that no section $s\in S_k$ vanishes at $x\in X$, then $\Delta_{Y_\bullet}(S_\bullet)=\{0\}$ is the origin for all flags $Y_\bullet$ centered at $\{x\}$.

Consider the series $(s\otimes \xi^{m_i})_{i\in \mathbb{N}}$ for which we have:
\begin{align}
\frac{\nu_{Y_\bullet}(s\otimes \xi^{m_i})}{k+m_i}= \frac{\nu_{Y_\bullet}(s)}{k+m_i} + \frac{ \nu_{Y_\bullet}(\xi^{m_i})}{k+m_i}  \to P \quad \text{ as } m_i  \to \infty.
\end{align}
But since $s\otimes \xi^{m_i} \in S^x_{m_i+k}$ we conclude the equality of the Newton-Okounkov bodies $\Delta_{Y_\bullet}(S^x_\bullet)=\Delta_{Y_\bullet}(S_\bullet)$.
In order to derive the equality of volumes, we will show that for both graded linear series the group generated by the semigroup of valuation points coincide. We trivially have an inclusion $G(\Gamma(S^x_\bullet))\subseteq G(\Gamma(S_\bullet))$. Now let $a\in G(\Gamma(S_\bullet))$ be an arbitrary element. We can write it as
\begin{align}
a= (\nu_{Y_\bullet}(\xi^1),k_1)- (\nu_{Y_\bullet}(\xi^2),k_2)
\end{align}
for some $\xi^i\in S_{k_i}$, $i=1,2$.
Choose again a section $s\in S_k$ such that $s(x)=0$ and note that $s\otimes \xi^i \in S^x_{k_i+k}$ for $i=1,2$. Then we can write
\begin{align}
a^\prime:&=(\nu_{Y_\bullet}(s\otimes \xi^1),k+k_1) -(\nu_{Y_\bullet}(s\otimes \xi^2),k+k_2) \\
&= (\nu_{Y_\bullet}(s),k)+ (\nu_{Y_\bullet}(\xi^1),k_1) -( ( \nu_{Y_\bullet}(s),k) + (\nu_{Y_\bullet}(\xi^2),k_2)) \\ 
&= a. 
\end{align}
But $a^\prime\in G(\Gamma(S^x_\bullet))$ which implies the equality of both groups and in particular the equality of both indices $\ind(S^x_\bullet)=\ind(S_\bullet)$. Applying the volume formula \eqref{volkk}, we get the desired equality of volumes.
\end{ex}

\subsection{Characterization of finitely generated graded linear series with full volume}
In this paragraph we want to classify all finitely generated graded linear series $S_\bullet$ corresponding to a big divisor $D$ such that $\vol(S_\bullet)=\vol(D)$.

We will need the following lemma.
\begin{lem}\label{lemfinbir}
Let $f\colon X\to Y$ be a dominant finite morphism of varieties. Then $f$ is of degree one if and only if $f$ is birational.
\end{lem}
\begin{proof}
If $f$ is birational, then there are open subsets $U=\spec B\subseteq X$ and $V=\spec A\subseteq Y$ such that $f_{|U}\colon U\to V$ is an isomorphism. This means that $A\cong B$ and thus $\mathbb{C}(X)=\fract(B)\cong \fract(A)=\mathbb{C}(Y)$. Hence, $f$ is of degree one.

Now let $f$ be finite of degree one. Consider an open affine subscheme $V=\spec A\subseteq Y$ such that $f^{-1}(V)=U=\spec B$. Then $f_{|U}\colon \spec B\to \spec A$ corresponds to an injective morphism of rings $\phi \colon A\to B$. Since $f$ has degree one, $\phi$ induces an isomorphism 
\begin{align}
\fract(\phi)\colon \fract(A)\to \fract(B).
\end{align}
Suppose $b_1,\cdots, b_n\in B$ is a set of generators of $B$ as an $A$-module. Let $a^\prime_1,\dots, a^\prime_n\in A$ be the set of denominators of the preimages of the $b_i$ under the isomorphism $\fract(\phi)$. Let $a^\prime=a^\prime_1\cdot \dots \cdot a^\prime_n$ be the product of the denominators and $A_{a^\prime}$ be the corresponding localization.
Next, we consider the morphism $\fract(\phi)$ restricted to the $A$-module $A_a$
\begin{align}
 \fract(\phi)_{|A_{a^\prime}} \colon A_{a^\prime}\to \fract (B).
\end{align} 
By construction, this restriction gives an isomorphism of $A_{a^\prime}$ to its image which is exactly $B_{a^\prime}$.
Applying the $\spec$ functor again gives  an isomorphism of schemes 
\begin{align}
f_{|U^\prime}\colon U^\prime=\spec B_{a^\prime} \to V^\prime=\spec (A_{a^\prime})=\text{D}(a^\prime).
\end{align}
Hence, $f$ is a birational morphism.
\end{proof}

\begin{thm} \label{thmcharempty}
Let $S_\bullet\subseteq T_\bullet$ be  finitely generated graded linear series corresponding to $D$. Let $T_\bullet$ be birational. Then the following two conditions are equivalent
\begin{enumerate}
\item $\vol(S_\bullet)=\vol(T_\bullet)$.
\item 
\begin{itemize}
\item The rational map $h_{S_\bullet}\colon X \dashrightarrow \proj (S_\bullet) $ is birational and
\item $\bs(S_\bullet)=\emptyset $ on $\proj(T_\bullet)$.
\end{itemize}
\end{enumerate}
\end{thm}
\begin{proof}
Consider the rational morphism corresponding to the section ring $R(T_\bullet)$
\begin{align}
h_{T_\bullet}\colon X \dashrightarrow \proj (T_\bullet)=:Y.
\end{align}
This is a rational contraction and we have $h_{T_\bullet}^*(\struc_Y(1))=D$ as well as $R(T_\bullet)\cong R_\bullet(Y,\struc_{Y}(1))$. Via this bijection, way we may regard $S_\bullet$ as a finitely generated graded linear subseries of $R_{\bullet}(Y,\struc_{Y}(1))$ on $Y$.
Since $\vol(T_\bullet)=\vol(\struc_{Y}(1))$, we can deduce from Theorem \ref{thmvolbase2}, that $\vol(S_\bullet)=\vol(T_\bullet)$ implies that $S_\bullet$ is base point free on $Y$. Therefore, we just need to show that $(a)$ is equivalent to the first part of $(b)$ under the assumption that $\vol(S_\bullet)=\vol(T_\bullet)$.
So let us assume that $S_\bullet$ is base point free on $Y$.
We want to show that the inclusion $\phi\colon R(S_\bullet)\to R(T_\bullet) $ induces a morphism $Y\to \proj(S_\bullet)=:Z$. Due to \cite[Remark 13.7]{GW}, the previous inclusion gives us a morphism  $G(\phi)\to \proj(S_\bullet)$ where \[G(\phi):=\bigcup_{s\in S_k, k> 0} D_{+}(s)\subseteq Y.\] But by the base point freeness of $S_\bullet$ on $Y$ we have $G(\phi)=Y$.
Hence, we get a globally defined morphism
\begin{align}
j\colon Y\to \proj(S_\bullet)=Z,
\end{align}
which fits into the following commutative diagram
\begin{align}
\xymatrix{ X \ar@{-->}[r]^{h_{T_\bullet}}\ar@{-->}[rd]_{h_{S_\bullet}}& Y\ar[d]^j \\
& Z.}
\end{align}
Since $h_{T_\bullet}$ is birational, the above diagram implies that $j$ is birational if and only if $h_{S_\bullet}$ is birational.
The morphism $j$ is by construction affine and projective and thus finite. By Lemma \ref{lemfinbir} and  taking the above equivalence into account, we get that $h_{S_\bullet}$ is birational if and only if  $j$ is finite of  degree one. But this is equivalent to 
\begin{align}
\vol(T_\bullet)=\vol(j^*\struc_{Z}(1))=\vol(\struc_{Z}(1))=\vol(S_\bullet).
\end{align}
\end{proof}
\begin{cor}
Let $D$ be a  semi ample big divisor on $X$ and $S_\bullet$ be a finitely generated graded linear series corresponding to $D$. Then the following two conditions are equivalent
\begin{enumerate}
\item $\vol(S_\bullet)=\vol(D)$
\item \begin{itemize}
\item $\bs(S_\bullet)=\bs(D)=\emptyset$
\item $h_{S_\bullet}\colon X\to \proj(S_\bullet)$ is birational.
\end{itemize}
\end{enumerate}
\end{cor}
\begin{proof}
The only thing which we need to prove in order to use the above Theorem is that b) implies that $\bs(S_\bullet)=\emptyset$ on $\proj(R_{\bullet}(X,D))$ but this follows from Lemma \ref{lemnotsmooth}.

\end{proof}
\subsection{Volume and base ideal} \label{sectbaseideal}
In this paragraph we want to take the scheme structure of the base locus into account. Hence, we will be interested in the connection between the volume and the base ideal of a graded linear series. The main motivation for this is  \cite[Theorem C]{jow}, which states that we can compute the volume of a birational graded linear series by passing to the base point free linear series on the blow-up along the base ideal. We will give a short introduction into base ideals and then derive some variations of Jow's statement.

Let $|V|$ be a linear series corresponding to a divisor $D$ on $X$.
Let $s_1,\dots, s_N\in V$ be global sections which induce a basis on $|V|$. Then we define $\f$ as the coherent sheaf generated by the $s_i$. More explicitly, $\f$ is the image sheaf of the following morphism:
\begin{align}
\struc_X^N\to \struc_X(D)
\end{align}
which is given on open sets $U\subseteq X$ by 
\begin{align}
(\lambda_1,\dots, \lambda_N)\mapsto \sum_{j=1}^N \lambda_j\cdot s_{j|U}.
\end{align} 
Consider an open cover $X=\bigcup_{i\in I} U_i$ and trivializations $\psi_i \colon \struc_X(D)_{|U_i} \cong \struc_{U_i}$. Then $\psi_i(\f_{|U_i})$ is an ideal of $\struc_{U_i}$, which does not depend on the choice of $\psi_i$. Gluing these ideals together induces a well defined ideal sheaf $\mathcal{I}$ such that $\f=\mathcal{I}\otimes_{\struc_X} \struc_X(D)$. Clearly, the support of $\mathcal{I}$ is $\bl(|V|)$ and we call it the base ideal of $|V|$. This gives  the set $\bl(|V|)$ the structure of a scheme. 
For a graded linear series $S_\bullet$, we denote the base ideal of $|S_k|$ by $\bsi_{S_k}$.
\begin{thm} \label{thmvolbase}

Let $S_\bullet$ and $T_\bullet$ be birational graded linear series corresponding to a divisor $D$.
Suppose $\bsi_{S_k}=\bsi_{T_k}$ for all $k\gg0$. Then
\begin{align}
\vol(S_\bullet)=\vol(T_\bullet).
\end{align}
\end{thm} 
\begin{proof}
An alternate formulation of \cite[Theorem C]{jow} is given in  \cite[Proposition 3.7]{H13}. The volume of $S_\bullet$ can be calculated in the following way. Let $\pi_k\colon X_k\to X$ be the blow-up of $X$ along $\bsi_{S_k}$. Let $M_k:=\pi_k^*D-E_k$, where $E_k$ is the exceptional divisor of the blow-up. Then
\begin{align}
\vol(S_\bullet)= \lim_{k\to\infty} \frac{(M_k)^d}{k^d} .
\end{align}
Hence, the volume of $S_\bullet$ just depends on the base ideals $\bsi_{S_k}$ for $k\gg0$ and the divisor $D$. Therefore  we get the same volume for $T_\bullet$ as for $S_\bullet$.
\end{proof}

\subsection{Rationality properties of finitely generated graded linear series}

The volume of a graded linear series  can in general behave very wildly (see for example \cite{KLM13}). Without any restrictions, it certainly can be irrational. Indeed, it is an easy consequence of \cite[Ex. 2.4.14]{laz} that actually all non-negative real numbers occur as the volume of some graded linear series. However, for a finitely generated divisor $D$ it is shown in \cite{AKL12} that there exists a flag $Y_\bullet$ such that the corresponding  Newton-Okounkov body $\Delta_{Y_\bullet}(D)$ is a rational simplex. We will generalize this result to the case of finitely generated birational graded linear series.
For finitely generated graded linear series which are not necessarily birational, we recover the well known fact that its volume is rational
\begin{thm}\label{thmgradvollb}
Let $S_\bullet$ be a birational graded linear series generated in degree one. Let $\pi \colon \tilde{X}\to X $ be the blow-up of $X$ along the base ideal $\bsi_{S_1}$ and $E$ be the exceptional divisor. Let $\tilde{Y}_\bullet$ be an admissible flag of $\tilde{X}$ centered at $\{\tilde{x}\}=\pi^{-1}(x)$ where $x\in X\setminus \bs(S_\bullet)$. Let $Y_\bullet$ be the admissible flag centered at $\{x\} $ given as the image of $\tilde{Y}_\bullet$ under $\pi $. Then  
\begin{align}
\Delta_{Y_\bullet}(S_\bullet)=\Delta_{\tilde{Y}_\bullet}(\pi^*D-E).
\end{align}
\end{thm}
\begin{proof}
First of all we show that the volume of both graded linear series are equal. Consider the graded linear series $\pi^*S_\bullet-E$ generated by $\pi^*S_1-E:=\{s/s_E^k \ : \ s\in S_1 \}$ corresponding to the divisor $\pi^*D-E$. By construction,$\pi^*S_\bullet-E$ is base point free and birational. Hence, due to Theorem \ref{thmvolbase}, we have 
\begin{align}
\vol(S_\bullet)=\vol(\pi^*S_\bullet-E)=\vol(\pi^*D-E).
\end{align}

The valuation $\nu_{Y_\bullet}$ respectively $\nu_{\tilde{Y}_\bullet}$ are both defined locally arround $x$ respectively $\tilde{x}=\pi^{-1}(x)$. Since $\pi$ defines an isomorphism arround $\{x\}$, we get $\nu_{Y_\bullet}(s)=\nu_{\tilde{Y}_\bullet}(\pi^*s)$ for all $s\in S_k$. As $E$ lies away from $x$, we also have 
$\nu_{\tilde{Y}_\bullet}(\pi^*s / s_E^k)=\nu_{\tilde{Y}_\bullet}(s)$ for all $s\in S_k$ and $s_E$ a defining section of $E$. This shows that 
\begin{align}
\Delta_{Y_\bullet}(S_\bullet)=\Delta_{\tilde{Y}_\bullet}(\pi^*S_\bullet-E)\subseteq \Delta_{\tilde{Y}_\bullet}(\pi^*D-E).
\end{align}
Combining this with the above equality of volumes, we get the desired result.
\end{proof}

\begin{cor}
Let $S_\bullet$ be  a finitely generated birational graded linear series. Then there is an admissible flag $Y_\bullet$ such that $\Delta_{Y_\bullet}(S_\bullet)$ is a rational simplex.
\end{cor}
\begin{proof}
Using the notation of Theorem \ref{thmgradvollb},
the only thing we need to prove is that there exists an admissible flag $\tilde{Y_\bullet}$ centered at some point $\tilde{x} \not \in E$ such that corresponding Newton Okounkov body of the globally generated divisor $\pi^*D-E$ is a rational simplex. But the existence of such a flag is proven in \cite[Proposition 7]{AKL12}. Note, that we used the fact that $\pi^*D-E$ is finitely generated since it is by construction free.
\end{proof}
If $S_\bullet$ is not birational but still finitely generated, we are not able to prove any rational polyhedrality property of the corresponding Newton- Okounkov body yet. However, the next theorem shows that the volume will nevertheless be rational.
\begin{thm}
Let $S_\bullet$ be a finitely generated graded linear series. Then the volume of $S_\bullet$ is a rational number.
\end{thm}
\begin{proof}
We may without loss of generality assume that $S_\bullet$ is generated in degree one.
The volume of $S_\bullet$ is equal to the volume of the free graded linear series $\pi^*S_\bullet-E$ where $\pi$ is the blow-up of $X$ along the base ideal of $S_1$. Hence, it suffices to show that the volume of a free graded linear series is rational.
So let us assume without loss of generality that $S_\bullet$ is a free finitely generated graded linear series generated in degree one on $X$ corresponding to a base point free divisor $D$ which is also generated in degree one.
If $S_\bullet$ is free on $X$, then by Lemma \ref{lemnotsmooth}, $S_\bullet$ is also free on $\proj(R_{\bullet}(X,D))$. As in the proof of
Theorem \ref{thmcharempty}, we conclude that the inclusion $R(S_\bullet)\subseteq  R_{\bullet}(X,D)$ induces a finite morphism:
\begin{align}
j\colon Y:=\proj(R_{\bullet}(X,D))\to \proj(S_\bullet)=:Z.
\end{align}
Let $k$ be the degree of $j$. Then we have 
\begin{align}
\vol(D)=\vol(\struc_Y(1))=\vol(j^*\struc_Z(1))=k\cdot \vol(\struc_Z(1))=k\cdot \vol(S_\bullet).
\end{align}
So the rationality of $\vol(S_\bullet)$ follows from the rationality of the volume of the free divisor $D$.
\end{proof}
Note, that the fact that the volume of a finitely generated graded linear series, can also be established by realizing that its Hilbert polynomial has indeed rational coefficients.

\subsection{Sheafification of  graded linear series}\label{sectionsheaf}
The aim of this section is to replace a birational graded linear series $S_\bullet$ by a possible larger one coming from global sections of coherent subsheaves of $\struc_X(kD)$. This sheafification process has the desirable feature that it does not change the volume of the graded linear series and hence not the Newton-Okounkov body.

Let $S_\bullet$ be a birational graded linear series. Denote by $\bsi_{S_k}$ the base ideal of the linear series $|S_k|$.

\begin{defi}
The \emph{sheafification} of $\Sb$ is given by the sheaf $\mathcal{S}_\bullet = (\mathcal{S}_k)_{k\geq 0}$ where 
\begin{align}
\mathcal{S}_k=\bsi_{S_k}\otimes_{\struc_X} \struc_X(kD).
\end{align}
The \emph{sheafified linear series} $\tilde{S}_\bullet$ is defined by 
\begin{align}
\tilde{S}_k=H^0(X,\mathcal{S}_k).
\end{align}
\end{defi}
\begin{rem}
The sheaf $\mathcal{S}_k$ is equal to $\f $ considered in  Section \ref{sectbaseideal} for $V=S_k$.
\end{rem}
\begin{thm}\label{thmsheafvol}
Let $\Sb$ be a birational graded linear series. Then the volume of $S_\bullet$ and the volume of the sheafified linear series $\tilde{S}_\bullet$ are equal i.e.
\begin{align}
\vol(\Sb)=\vol ( \tilde{S}_\bullet).
\end{align}
\end{thm}
\begin{proof}
This is a consequence of Theorem \ref{thmvolbase}. The sheaf $\mathcal{S}_k=\bsi_{S_k}\otimes_{\struc_X} \struc_X(D)$ is by construction globally generated and therefore the base ideal of $\tilde{S}_k$ is equal to $\bsi_{S_k}$.
Thus, the base ideal of $\tilde{S}_k$ and of $S_k$ are both equal for every $k>0$, from which  the equality of volumes follows.
\end{proof}

\begin{cor} \label{sheafvol}
Let $\Sb$ be a birational graded linear series. Then for all admissible flags $Y_\bullet$ we have:
\begin{align}
\Delta_{Y_\bullet}(S_\bullet)=\Delta_{Y_\bullet}(\tilde{S}_\bullet)
\end{align}
\end{cor}
\qed

\begin{ex}
Consider 
\[X=\mathbb{P}^2=\proj \mathbb{C}[X_1,X_2,X_3]\] and $L=\struc(2)$. We can identify \[H^0(X,L)\cong \mathbb{C}[X_1,X_2,X_3]_2.\] Now let $S_\bullet$ be the graded linear series which is generated by all monomials of degree two except $X_2X_3$, i.e. generated by
\[
S_1:= span_{\mathbb{C}}( X_1^2,X_2^2, X_3^2, X_1X_2, X_1X_3)
\]
 Clearly $S_\bullet$ is base point free and $\Delta_{Y_\bullet}(D)=\Delta_{Y_\bullet}(S_\bullet)$, for the standard flag $Y_\bullet$ such that 
 \[Y_1:=\text{V}(X_1), 
Y_2:= \text{V}(X_1)\cap \text{V}(X_2) \text{ and } Y_3:= \{ [0:0:1] \}. \]
Next, we want to show that the semigroup $\Gamma(S_\bullet)$ generates $\mathbb{Z}^3$ as a group. 
We have \[v_1:=(1,1,1)=(\nu(X_1X_2),1)\in \Gamma(S_\bullet),\]  \[v_2:=(1,0,1)=(\nu(X_1X_3),1)\in \Gamma(S_\bullet),\] and \[v_3:=(0,2,1)=(\nu(X_2),1)\in \Gamma(S_\bullet).\]
This shows that $e_2:= v_1-v_2\in G(\Gamma(S_\bullet))$, $e_3= v_3-2e_2\in G(\Gamma(S_\bullet))$ and $e_1=v_1-e_2-e_3\in G(\Gamma(S_\bullet))$. Hence,  $G(\Gamma(S_\bullet))=\mathbb{Z}^3$. From this fact and the equality of Newton-Okounkov bodies, we deduce that $\vol(S_\bullet)=\vol(D)$. Applying Theorem \ref{thmvolbase} gives us that $S_\bullet$ is birational and that $\bsi_{S_1}=\struc_X$. Hence, from the above theorem we expect the following equality for the sheafified linear series.
\begin{align}
R(\tilde{S}_\bullet)= R_{\bullet}(X,L).
\end{align}
Indeed, we will show by hand how the missing global section $X_2X_3\in H^0(X,L)$  can be glued locally from sections in $S_1$.

In \[U_i:= \spec \mathbb{C}[X_1/X_i,X_2/X_i, X_3/X_i]\subseteq \mathbb{P}^2 \]
 we have $s_i=(X_2X_3)_{|U_i}= \frac{X_2X_3}{X_i^2}\in H^0(U_i,\mathcal{O}_X)$.

Since $H^0(U_i,\mathcal{S}_1)$ is a $H^0(U_i,\mathcal{O}_{X})$-module, we have
\begin{align}
s_i= \lambda_i \cdot (t_i)_{|U_i} \in H^0(U_i,\mathcal{S}_1)
\end{align}
for $\lambda_i=s_i$ and $t_i=X_i^2$.

\end{ex}

\section{Slice formula for graded linear Series}
Slice formulas are one of the most important tools to analyze the shape of Newton--Okounkov bodies. They are the main ingredient for the characterization of Newton-Okounkov bodies on surfaces \cite[Theorem 6.4]{LM09}, for Jow's Theorem \cite{jow} and for most of the rational polyhedrality properties (see e.g. \cite{AKL12}, or \cite{SS17}).

By a slice formula we mean the following: Let $S_\bullet$ be a graded linear series and $Y_\bullet$ be a flag. A slice of the Newton-Okounkov body $\Delta_{Y_\bullet}(S_\bullet)$ is given by intersecting it with some affine hypersurface $\{t\}\times \mathbb{R}^{d-1}\subseteq \mathbb{R}^{d}$ for $t\geq 0$. We denote this slice by $\Delta_{Y_\bullet}(S_\bullet)_{\nu_1=t}$. Suppose there is a graded linear series $W_\bullet$ on $Y_1$ such that 
\begin{align}
\Delta_{Y_\bullet}(S_\bullet)_{\nu_1=t}= \Delta_{Y^\prime_\bullet}(W_\bullet)
\end{align}
where $Y^\prime_\bullet$ is the restriction of the flag $Y_\bullet$ on $Y_1$.
Then the above equality will be called a slice formula.

Slice formulas are very useful since  the set of all slices $\Delta_{Y_\bullet}(S_\bullet)_{\nu_1=t}$ determines $\Delta_{Y_\bullet}(S_\bullet)$ and thus we are able to reduce the calculation of $\Delta_{Y_\bullet}(S_\bullet)$ to the calculation of Newton-Okounkov bodies of dimension one less. This will give us  in some cases the possibility to argue inductively on the dimension of $X$.

In all cases we will consider, the graded linear series $W_\bullet$ will be a certain \emph{restricted linear series}. Let us quickly recall basic notions of restricted linear series.
Let $S_\bullet$ be a graded linear series on $X$ corresponding to $D$ and let $Y\subseteq X$ be an irreducible subvariety. Let
\begin{align}
\text{rest}_k\colon H^0(X,\struc_{X}(kD))\to H^0(Y,\struc_{Y}(kD))
\end{align}
be the restriction of global sections of $kD$ to $Y$.
Then we define the restricted graded linear series $S_{|Y,\bullet}$ on $Y$, by
\begin{align}
S_{|Y,k}:= \text{rest}_k(S_k).
\end{align}
If $R(S_\bullet)=R_\bullet(X,D)$, we write $R_{\bullet}(X,D)_{|Y}$ for the graded algebra of sections of the restricted linear series of $S_\bullet$.
We will write $\vol_{X|Y}(S_\bullet):=\vol_{Y}(S_{|Y,\bullet})$ for the volume of the restricted linear series, as well as $\vol_{X|Y}(D)$ if $R(S_\bullet)=R_\bullet(X,D)$.
Let $Y_\bullet$ be a flag on $X$. We write 
\begin{align}
\Delta_{X|Y_1}(S_\bullet):= \Delta_{Y^\prime_\bullet}(S_{|Y_1,\bullet})
\end{align}
for $Y^\prime_\bullet$ defined as above.
In a lot of situations we want to assume that $Y_1$ is a Cartier divisor. Therefore, we make the following definition.
\begin{defi}
Let $X$ be a projective variety and $Y_\bullet$ an admissible flag. We call $Y_\bullet$ \emph{very admissible} if $Y_1$ defines a Cartier divisor.
\end{defi}
Note that if $X$ is smooth then every admissible flag is very admissible.

The existence of slice formulas is directly connected to the distribution of valuative points in $\Delta_{Y_\bullet}(S_\bullet)$. 
Therefore, before deriving slice formulas, we will analyze this distribution.

\subsection{Valuation points}
A different construction of the Newton-Okounkov body than the one we mentioned is the following. First on constructs the set of normalized valuations  $\Sigma:=\bigcup_{m>0} 1/m \cdot  \Gamma_m(S_\bullet)$ and then one takes the closed convex hull of this set. This is completely equivalent to the earlier mentioned definition of Newton-Okounkov bodies.
 A priori, it might happen that taking the convex hull destroys a lot of information about $\Sigma$. In this paragraph we will see that this is not the case and indeed it is not even necessary to take the convex hull in the first place.

\begin{defi}
Let $S_\bullet$ be a graded linear series. Let $Y_\bullet$ be an admissible flag. Then a point $P\in \Delta_{Y_\bullet}(D)$ is called a  \emph{valuation point} or \emph{valuative} if there is a section $s\in S_k$ such that $\nu_{Y_\bullet}(s)/k=P$ for some $k\in\mathbb{N}$.
\end{defi}

\begin{lem}
Let $S_\bullet$ be a graded linear series on $X$. Let $Y_\bullet$ be an admissible flag. Let $P$ and $Q$ be two valuative points of $\Delta_{Y_\bullet}(S_\bullet)$. Then all rational points in the line segment $PQ$ are also valuative.
\end{lem}
\begin{proof}
Let $s\in S_k$ and $t\in S_m$ be sections such that $P=\nu_{Y_\bullet}(s)/k$ and $Q=\nu_{Y_\bullet}(t)/m$.
Let $A:= (a/b)\cdot P + (1-a/b)\cdot Q$ for integers $a,b$ such that $0\leq a<b$ be an arbitrary rational point in the line segment $PQ$.
Define the section
\begin{align}
s^{ma}\otimes t^{k(b-a)}\in S_{kbm}.
\end{align}
Then we have
\begin{align}
\frac{\nu_{Y_\bullet}(s^{ma}\otimes t^{k(b-a)})}{kbm} &= ma\cdot \frac{\nu_{Y_\bullet}(s)}{kbm}+ k(b-a)\cdot \frac{\nu_{Y_\bullet}(t)}{kbm}\\
&= a/b\cdot P +(1-a/b)\cdot Q\\
&=A.
\end{align}
 This shows that every rational point in the seqment $PQ$ is valuative.

\end{proof}

\begin{lem}\label{lemval}
Let $S_\bullet$ be a graded linear series on $X$.
Let $Y_\bullet$ be an admissible flag. Let $P_1,\dots, P_n$ be valuative points of $\Delta_{Y_\bullet}(S_\bullet)$. Then all rational points in the relative interior of the convex hull of $P_1,\dots, P_n$ are also valuative.
\end{lem}
\begin{proof}
We will prove this by induction on the number of points $n$. The case $n=2$ was done in the previous lemma. Let us prove the claim for $n$, assuming that it holds for integers $\leq n-1$. 
Let $P$ be a rational point in the interior of the convex hull.
The induction hypothesis tells us that all rational points on the facets of the convex hull are valuative. Now, consider the line  going through $P_1$ and $P$. This line intersects the boundary of the convex hull in $P_1$ and in one more point, which we will call $Q$. Clearly the point $Q$ is rational and hence, by the induction hypothesis, valuative. Now, by construction, $P$ is in the convex hull of the two valuative points $P_1$ and $Q$. Again, using the induction hypothesis, the point $P$ is valuative.
\end{proof}

\begin{thm}\label{thmvalrat}
Let $S_\bullet$ be a graded linear series on $X$. Let $Y_\bullet$ be an admissible flag. Then all rational points in the relative interior of the Newton-Okounkov body $\Delta_{Y_\bullet}(S_\bullet)$ are valuative.
\end{thm}
\begin{proof}
%Let us assume that the valuational points are not dense in the Okounkov body $\Delta(S)$. Then there is an open set $U\subset \Delta(S)$ which contains not valuational points. We can by shrinking $U$ if necessary assume that the closure of $U$ is still contained in $\Delta(S)$. 
Let $P$ be a rational point in the interior of $\Delta_{Y_\bullet}(S_\bullet)$. By construction of $\Delta_{Y_\bullet}(S_\bullet)$ this means that
\begin{align}
P\in \conv \left(\bigcup_{m>0}  1/m\cdot \nu_{Y_\bullet}(S_m\setminus\{0\}) \right).
\end{align}
So there are finitely many sections $s_i\in S_{m_i}$ and coefficients $\lambda_i\geq 0$ for $i=1,\dots, N$ such that $\sum_{i=1}^N \lambda_i=1$ and 
\begin{align}
P= \sum_{i=1}^{N} \lambda_i \cdot \nu_{Y_\bullet}(s_i)/m_i.
\end{align}
Define $K:=\prod_{i=1}^N m_i$. Then $s_i^{K/m_i}\in S_K$ and we can write
\begin{align}
P=\sum_{i=1}^{N} \lambda_i \cdot \nu_{Y_\bullet}(s_i^{K/m_i})/K.
\end{align}
This implies that \[
P\in \conv \left( 1/K\cdot \nu_{Y_\bullet}(S_K\setminus\{0\})\right),\]
 and due to the previous lemma it follows that $P$ is valuative.

%By the definition of $U$ there is an integer $k$ such that $U\subseteq conv (\Gamma(S_i)$. But the previous Lemma tells us that all rational points in $\Gamma(S_i)$ are valuative. Since $U$ is an open subset it must contain rational points. Thus we have a contradiction.
\end{proof}
\begin{rem}
The fact that the valuative points are dense in the Newton-Okounkov body was already proven in \cite{KL14}  for the divisorial case $S_\bullet=R_{\bullet}(X,D)$ and in the general case of graded linear series in \cite[Lemma 2.6]{KMS12}.
\end{rem}
\begin{cor}
Let $D$ be a big divisor on $X$ and let $Y_\bullet$ be a very admissible flag such that $Y_1\not \subseteq \bsp(D)$. Then all the rational points in the relative interior of $\Delta_{Y_\bullet}(D)_{\nu_1=0}$ are valuative.
\end{cor}
\begin{proof}
By  construction of the Newton-Okounkov body, the valuative points of the slice $\Delta_{Y_\bullet}(D)_{\nu_1=0}$ are of the form $(0,P)$ where $P$ is a valuative point of the Newton-Okounkov body of the restricted linear series $R_{\bullet}(X,D)_{|Y_1}$, which we denote by $\Delta_{X|Y_1}(D)$.
But due to the slice formula in \cite[Theorem 4.24 b]{LM09}  we have an equality
\begin{align}
\Delta_{X|Y_1}(D)= \Delta_{Y_\bullet}(D)_{\nu_1=0}.
\end{align}
Combining this with the above theorem yields  the desired result.

\end{proof}

The fact that some points on the boundary of the Newton-Okounkov body are valuative and some may just be limits of valuative points corresponds to the fact that the semigroup $\Gamma(S_\bullet)$ may be not finitely generated. 
Finite generation of the semigroup is a very pleasant property. It was shown in \cite{A13} that if $\Gamma(S_\bullet)$ is finitely generated then there exists a corresponding flat degeneration of $X$ to the toric variety whose normalization corresponds to the polytope $\Delta_{Y_\bullet}(S_\bullet)$

The connection to the existence of valuative points is given by the following theorem.

\begin{thm}\label{thmvalfin}
Let $S_\bullet$ be a graded linear series,  $Y_\bullet$ be  an admissible flag of $X$ and $\Gamma(S_\bullet)$ be finitely generated. Then all rational points of $\Delta_{Y_\bullet}(S_\bullet)$ are valuative. If $\Delta_{Y_\bullet}(S_\bullet)$ is rational polyhedral, then $\Gamma(S_\bullet)$ is finitely generated if and only if all rational points of $\Delta_{Y_\bullet}(S_\bullet)$ are valuative. 
\end{thm}
\begin{proof}
Let $\Gamma(S_\bullet)$ be finitely generated and $(\nu_{Y_\bullet}(s_1),k_1),\dots, (\nu_{Y_\bullet}(s_N),k_N)$ be the generators. Then  $\Delta_{Y_\bullet}(S_\bullet)$ is equal to the convex hull of the points $P_1=1/k_1 \cdot \nu_{Y_\bullet}(s_1),\dots, P_N=1/k_N\cdot \nu_{Y_\bullet}(s_N)$. Due to Lemma \ref{lemval}, all points in $\Delta_{Y_\bullet}(S_\bullet)$ are valuative.

 Let now  $\Delta_{Y_\bullet}(S_\bullet)$ be rational polyhedral. It remains to prove that $\Gamma(S_\bullet)$ is finitely generated if all rational points are valuative. However, this follows from  \cite[Corollary 2.10]{BG09}.
\end{proof}
\begin{rem}
Note that in the second part of the above theorem the condition that $\Delta_{Y_\bullet}(S_\bullet)$ is rational polyhedral is necessary. For a counterexample consider \cite[Proposition 1.17]{LM09} for any $K$ which is not rational polyhedral. In this case it is clear from the construction that all rational points are valuative. However, the corresponding semigroup can never be finitely generated unless its Newton-Okounkov body is rational polyhedral.
\end{rem}

The next theorem says that for a surface $X$ the semigroup $\Gamma(D)$ is almost never finitely generated.
\begin{thm}\label{thmfingen}
Let $X$ be a smooth surface.
Let $S_\bullet$ be a graded linear series corresponding to a big divisor $D$ such that $\vol(S_\bullet)=\vol(D)$. Let $C\subseteq X$ be  a curve of genus $g>0$. Then for a general point $x\in C$ the semigroup $\Gamma_{C\supseteq\{x\}}(S_\bullet)$ is not finitely generated.
\end{thm}
\begin{proof}
Without loss of generality, we may replace $D$ by $kD$ and can therefore assume that there is a non-negative integer $t\in \mathbb{N}$  such that $\vol_{X|C}(D-tC)> 0$. Consider the Zariski decomposition of $D-tC$:
\begin{align}
D_t:= D-tC= P_t + N_t,
\end{align}
and choose $x\in C$ very general such that the semigroup $\{(k,\operatorname{ord}_x(s))\ |\ s\in H^0(C,\struc_C(P_{t})) \}$ is not finitely generated \cite[Example 1.7]{LM09}.
 If we set $\Delta_{\{x\}}(P_{t|C})=[0,c]$, then the failure of finite generation just means that $c$ is not a valuative point of $\Delta_{\{x\}}(P_{t|C})$. 
But since $\vol_{X|C}(D_t)=\vol_{X|C}(P_t)>0$, we deduce that $C\not\subseteq \bsp(P_t)$ and we have $\vol_{X|C}(P_t)=\vol(P_{t|C})$ \cite[Corollary 2.17]{ELMNP09}. Thus $c$ is not a valuative point of the restricted Newton-Okounkov body $\Delta_{X|C}(P_t)$. The valuative points of $\Delta_{X|C}(P_t)$ correspond to the valuative points of the restricted Newton-Okounkov body $\Delta_{X|C}(D_t)$ up to a translation of $\operatorname{ord}_x (N_{t|C})$. But each valuative point $Q$  of $\Delta_{X|C}(D_t)$ corresponds one to one to the valuative point $(t,Q)$ of $\Delta_{C\supseteq\{x\}}(D)_{\nu_1=t}$. This shows that $(t,c+\operatorname{ord}_x(N_{t|C}))\in \Delta_{C\supseteq \{x\}}(D)$ is not a valuative point and thus surely it is not a valuative point of $\Delta_{C\supseteq \{x\}}(S_\bullet)$. Applying Theorem \ref{thmvalfin} gives then the desired failure of finite generation.
 
\end{proof}

\begin{ex}

Let $X$ be a smooth Mori dream surface, let $D$ be a big  divisor on $X$ and $Y_\bullet\colon X\supset C\supset \{x\}$ be an admissible flag on $X$ consisting  of a curve $C$ on $X$ which is not contained in $\bsp(D)$.
We use \cite[Theorem B]{KLM12} for describing the Newton-Okounkov body of a  big divisor on a surface:
There  are piecewise linear functions with rational slopes and rational breaking points  $\alpha,\beta \colon [\nu,\mu]\to \mathbb{R}^{+}$ such that the Newton-Okounkov-body is given by:
\begin{align} \label{oksurf}
\Delta_{Y_\bullet}(D)= \{ (t,y)\in \mathbb{R}^2: \nu\leq t\leq \mu, \quad \text{and } \alpha(t)\leq y \leq \beta(t) \}
\end{align}
Moreover, the number $\nu$ is rational and $\mu$ is given by
\begin{align}
\mu:= \sup \{ s>0: D-sC \text{ is big} \}.
\end{align}
Since in our case $X$ is a Mori dream space and thus the Big cone is rational polyhedral, the number $\mu$ is rational as well. In this situation we have a quite good understanding of the valuative points of the Newton-Okounkov body.
The following points are valuative:
\begin{enumerate}

\item the rational points in the interior of $\Delta_{Y_\bullet}(D)$
\item points of the form $(\nu,y)$ for rational $y\in [\alpha(\nu),\beta(\nu))$
\item points of the form $(t,\alpha(t))$ for all rational  $t\in [\nu,\mu)$ .
\end{enumerate}

Let us  prove that all the above listed points are indeed valuative.

Part (a) follows from Theorem \ref{thmvalrat}. Part (b) follows by using the slice formula \cite[Theorem 4.24 b)]{LM09} which states that  $\Delta_{Y_\bullet } (D)_{\nu_1=t}=\Delta_{X|C}(D-tC)$ for all $t\in [\nu,\mu)$. Indeed, for such rational $t$ the valuative points  of the latter Newton-Okounkov body correspond to the valuative points of $\Delta_{Y_\bullet}(D)$ with first coordinate equal to $t$.
Hence, again by  Theorem \ref{thmvalrat}, all the rational points of the form $(\nu,t_1)$ and  for $t_1\in (\alpha(\nu),\beta(\nu))$  are valuative. 
Part (c) follows from the following fact:
Let $S_\bullet$ be a finitely generated graded linear series on a curve $C$ and let $P$ be a smooth point on $C$. Let $\Delta(S_\bullet)=[b,c]$. Then $b$ is a valuative point.
To prove this we can without loss of generality assume that $S_\bullet$ is generated in degree one. Let $s_1,\dots, s_l\in S_1$ be  the generators of $S_\bullet$. Now suppose $b$ is not a valuative point. Then $\nu(s_i)\geq b+1$ for all $i=1,\dots  , l$. Consider $s\in S_k$ which can be written as $s=\sum_{\alpha\in\mathbb{N}^l} c_{\alpha} \overline{s}^\alpha$ where $\overline{s}=(s_1,\dots, s_l)$. Then
\begin{align}
\nu(s)&=\nu(\sum_{\alpha\in\mathbb{N}^l} c_{\alpha} \overline{s}^\alpha)\\
&\geq \min (\nu(\overline{s}^\alpha)\\
&\geq k(b+1)
\end{align}
which implies  that $b$ does not lie in $\Delta(S_\bullet)$ inducing a contradiction. 
Using this fact for the restricted graded linear series of $D-tC$ to $C$ which is indeed finitely generated since $X$ is a Mori dream space, gives us the valuativity of the remaining listed points.
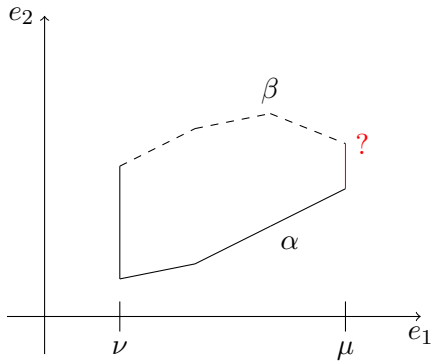
\begin{figure}[h] 
\centering
\begin{tikzpicture}
\draw[->] (-.5,0)--(5,0)node[below]{$e_1$};
\draw[->] (0,-.5)--(0,4) node[left]{$e_2$};

\draw[style=dashed] (1,2)-- (2,2.5);
\draw[style=dashed]  (2,2.5)-- (3,2.7)node[above]{$\beta$};
\draw[style=dashed]  (3,2.7)--(4,2.3);
\draw (1,0.5)--(1,2);
\draw (1,0.5)-- (2,0.7);
\draw (2,0.7)--(3,1.2)node[below right]{$\alpha$};
\draw (3,1.2)-- (4,1.7);

\draw[color=red] (4,1.7)-- (4,2.3)node[right] {$?$};
\draw (1,0.2)--(1,-0.2) node[below] {$\nu$};
\draw (4,0.2)--(4,-0.2) node[below] {$\mu$};
\end{tikzpicture}
\caption{Valuation points of NO-body on a surface}
\label{okounkov}
\end{figure}

If $C$ is a curve of genus $g>0$ and $x$ is a very general point in $C$, then we can say even more.
The points of the form $(t,\beta(t))$ for $t\in [\nu,\mu)$ are not valuative if $\beta(\nu)>\alpha(\nu)$.
If $\alpha(\nu)=\beta(\nu)$, then this holds for $t\in (\nu,\mu)$.
In order to prove this, we make use of the proof of Theorem \ref{thmfingen}. There we showed that for rational $t\in [\mu,\nu)$ such that $\text{vol}(\Delta_{X|C}(D-tC))>0$ for a general choice of $x\in C$ the point $(t,\beta(t))$ is not valuative. Since $t$ varies in a countable set, we conclude that for a very general choice of $x\in C$ this holds for all considered rational $t$ at once.

In this situation the only points where we  do not know whether they are  valuative or not are the rational points of the form $(\mu, y) $ for $y\in [\alpha(\mu),\beta(\mu)]$.
The situation is summarized in Figure \ref{okounkov}.
\end{ex}

\subsection{Slice Formula}
In this paragraph we generalize the slice formula given in \cite[Theorem 4.2.4]{LM09} to graded linear series $S_\bullet$. Let us first state the content of the theorem:
Let $D$ be a big divisor, $Y_\bullet$ be an admissible flag such that $Y_1$ is an effective Cartier divisor for which $Y_1\not\subseteq \bsp(D)$ and $\mu:=\sup \{t\in \mathbb{R}^{+}| (D-tY_1) \text{ is big}\}$. Then we have for all $0\leq t < \mu$
\begin{align}
\Delta_{Y_\bullet}(D)_{\nu_1=t}= \Delta_{X|Y_1}(D-tY_1).
\end{align}

The following definition will be useful for the generalization.
\begin{defi}
Let $S_\bullet$ be a graded linear series on $X$. Let $Y\subseteq X$ be an irreducible subvariety of codimension $1$ which defines a Cartier divisor and $\varepsilon$ be a non-negative rational number. Then we define the graded linear series $S_\bullet-\varepsilon Y$ by setting
\begin{align}
(S_\bullet-\varepsilon Y)_k=\{ s/s_Y^{\lceil \varepsilon\cdot k \rceil} : s\in S_k \quad \operatorname{ord}_{Y}(s)\geq  \lceil \varepsilon\cdot k\rceil \} \subseteq H^0(X, \struc_X(kD-\lceil \varepsilon\cdot k \rceil Y))).
\end{align}
\end{defi}

Using the above definition, we are able to formulate our first slice formula for slices which meet the interior of the corresponding Newton-Okounkov body.
\begin{thm}\label{thmslicepositive}
Let $S_\bullet$ be a graded linear series. Let $Y_\bullet$ be a very admissible flag and $\varepsilon$ a positive rational number such that $\{\varepsilon\}\times\mathbb{R}^{d-1}$ meets the interior of $\Delta_{Y_\bullet}(S_\bullet)$.Then 
\begin{align}
\Delta_{Y_\bullet}(S_\bullet)_{\nu_1=\varepsilon}=\Delta_{X|Y_1}(S_\bullet -\varepsilon Y_1)
\end{align}
via the identification of $\{\varepsilon\}\times \mathbb{R}^{d-1}\cong \mathbb{R}^{d-1}$.
\end{thm}
\begin{proof}
By considering the $k$-th Veronese $S^{(k)}_\bullet$ of the graded linear series $S_\bullet$ for a high enough multiple, i.e. $S^{(k)}_\bullet$ defined by $S^{(k)}_l=S_{l\cdot k}$, we can without loss of generality assume that $\varepsilon$ is an integer.
We will now show that the rational points in the interior of both Newton-Okounkov bodies are indeed equal, from which the statement will follow by Theorem \ref{thmvalrat}.

Consider first the rational points in the interior of $\Delta_{X|Y_1}(S_\bullet -\varepsilon Y_1)$. By construction these are given in degree $k$ by
\begin{align}
1/k 
\cdot \Gamma_k((S_\bullet -\varepsilon Y_1)_{|Y_1})&=\{1/k\cdot (\nu_2(s),\dots,\nu_d(s))\ |\ s\in S_k \text{ s. t. } \nu_1(s)=\varepsilon \cdot k \}\\
 &\cong  \{1/k\cdot(\varepsilon \cdot k,\nu_2(s),\dots,\nu_d(s))\ | \ s\in S_{k} \text{ s.t. }  \nu_1(s)=\varepsilon \cdot k \}\\
 &= \left(1/k\cdot \Gamma_{k}(S_\bullet)\right)\cap (\{\varepsilon\} \times \mathbb{R}^{d-1}). 
\end{align}
But the last set are just the valuative points of $\Delta_{Y_\bullet}(S_\bullet)_{\nu_1=\varepsilon}$ in degree $k$. This finishes the proof.

\end{proof}
\begin{cor}\label{corslicepos}
Let $S_\bullet$ be a graded linear series. Let $Y_\bullet$ be a very admissible flag and $\varepsilon$ a positive rational number such that $\{\varepsilon\}\times \mathbb{R}^{d-1}$ meets the interior of $\Delta_{Y_\bullet}(S_\bullet)$.Then 
\begin{align}
\Delta_{Y_\bullet}(S_\bullet)_{\nu_1\geq \varepsilon}:=\Delta_{Y_\bullet}(S_\bullet)\cap [\varepsilon,\infty) \times \mathbb{R}^{d-1}= \Delta_{Y_\bullet}(S_\bullet -\varepsilon Y_1)+ (\varepsilon,0,\dots,0).
\end{align}
\end{cor}
\begin{proof}
This follows by realizing that the slices of both sides agree for all rational vertical slices. Indeed, we have for all $\delta>0$ such that 
$\{\delta+\varepsilon\}\times \mathbb{R}^{d-1}$ meets the interior of $\Delta_{Y_\bullet}(S_\bullet)$.
\begin{align}
\Delta_{Y_\bullet}(S_\bullet -\varepsilon Y_1)_{\nu_1=\delta}&=\Delta_{X|Y_1}(S_\bullet-(\varepsilon+\delta)Y_1)\\
 &= \Delta_{Y_\bullet}(S_\bullet)_{\varepsilon+\delta}.
\end{align}
\end{proof}

The above theorem shows that for $t>0$ the slice formula of \cite[Theorem 4.2.4]{LM09} completely generalizes to the case of arbitrary graded linear series without any restrictions. However, the reduction to the case $t=0$ does not work as in \cite{LM09}. The idea of the proof was  to replace the divisor $D$ by some small perturbation $D+\varepsilon Y_1$ and thus reduce the question to the case $t>0$. However, for a graded linear series it is not clear how to generalize this construction. Therefore, we need some additional  properties for the graded linear series $S_\bullet$ in order to recover more of the geometry of $X$ and the corresponding divisor $D$. We would like to assume that $S_\bullet$ as well as the restricted series $S_{|Y,\bullet}$ are birational. In order to make sure that the restricted series has this property, we pose a stronger condition on $S_\bullet$, namely, that it contains an ample series. (This corresponds to condition (C) in \cite{LM09}, see Definition \ref{defample}).

In addition, we will start with the case that  $S_\bullet$ is also  finitely generated and $\vol(S_\bullet)=\vol(D)$.
After that we will reduce the general case to the special case by using Fujita approximation.

\begin{lem}\label{lemprojrest}
Let $S_\bullet\subseteq T_\bullet$ be two birational finitely generated graded linear series such that the  map of projective spectra $\proj(T_\bullet)\to \proj(S_\bullet)$  defined by the inclusion of graded linear algebras $R(S_\bullet)\subseteq R(T_\bullet)$ is globally defined. Then for each closed subvariety $Y\subseteq X$, the induced map $\proj(T_{|Y,\bullet}) \to \proj(S_{|Y,\bullet})$ is also globally defined.
\end{lem}
\begin{proof}

Consider the following diagram of morphisms of graded algebras
\begin{align}
\xymatrix{  R(S_\bullet) \ar@{->>}[r]^{r_S} \ar@{^{(}->}[d]^{\iota} 
&  R(S_{|Y,\bullet}) \ar@{^{(}->}[d]^{\iota_Y} \\ 
R(T_\bullet) \ar@{->>}[r]^{r_T} &
R(T_{|Y,\bullet})
            }  
\end{align}
where the horizontal mappings are just the restriction of sections and the vertical maps are given by inclusion.
For a graded algebra $U=\bigoplus_{k\in\mathbb{N}} U_k$, define $U^{+}:=\bigoplus_{k>0} U_k$.
We want to show that the inclusion $\iota_Y$ defines a global morphism of corresponding projective spectra.
Therefore, we need to check if the preimage under $\iota_Y$ of each relevant homogeneous prime ideal $\p \subset R(T_{|Y,\bullet})^{+}$ is still relevant.
So suppose that the preimage is not relevant, i.e. $R(S_{|Y,\bullet})^{+}\subseteq \iota_Y^{-1}(\p)$. Then by definition of the restriction morphism, we get:
\begin{align}
R(S_\bullet)^{+}=r_S^{-1}(R(S_{|Y,\bullet})^{+})\subseteq r_S^{-1}(\iota_Y^{-1}(\p)).
\end{align}
This means that the ideal on the right hand side is not relevant. Due to the commutativity of the above diagram, the right hand side is equal to $\iota^{-1}(r_T^{-1}(\p))$. However, the ideal $r_T^{-1}(\p)$  is relevant since $r_T$ is surjective and therefore the ideal $ \iota^{-1}(r_T^{-1}(\p))$ is relevant as well since, by assumption, $\iota$ induces a global morphism of projective spectra. Hence, we get a contradiction, which shows the claim
\end{proof}
Let us now define what it means to contain an ample series.
\begin{defi}\label{defample}
Let $S_\bullet$ be a graded linear series on $X$ corresponding to $D$. We say that \emph{$S_\bullet$ contains the ample series $A=D-E$} if
\begin{enumerate}
\item $S_k\neq 0$ for $k\gg 0$ and
\item there is a decomposition of $\mathbb{Q}$-divisors $D=A+E$ where $A$ is ample and $E$ is effective such that
\begin{align}
H^0(X,\mathcal{O}(kA))\subseteq S_k \subseteq H^0(X,\mathcal{O}(kD)
\end{align}
for all $k$ divisible enough. Note that the inclusion of the outer groups is given by the multiplication of a defining section of $kE$.
\end{enumerate}
\end{defi}
As it was already pointed out in \cite{jow}, it is not difficult to show that a graded linear series containing an ample series is birational. This follows from the birationality of the ample series.

\begin{lem}\label{lembirrest}
Let $S_\bullet$ be a graded linear series corresponding to $D$ which contains the ample series $D-E$. Let $Y\subseteq X$ be a closed irreducible subvariety  such that $Y\not \subseteq \text{Supp}(E)$. Then the restricted linear series $S_{|Y,\bullet}$ contains an ample series corresponding to the decomposition $D_{|Y}=A_{|Y}+E_{|Y}$. 
\end{lem}
\begin{proof}
The restriction of an ample divisor to a closed subvariety is ample. Since $Y\not\subset E$ we conclude that $E_{|Y}$ is effective.
Hence $D_{|Y}=A_{|Y}+E_{|Y}$ is a decomposition into ample and effective. Furthermore, the stable base locus of $S_\bullet$ is contained in $E$, by the assumption that $S_\bullet$ contains the ample series $D-E$.  Hence, there is a $k\gg 0$ such that $S_{|Y,k}\neq 0$.
For $k$ divisible enough, we conclude, by Serre vanishing, that $H^0(Y,\mathcal{O}_Y(kA))=H^0(X,\mathcal{O}_X(kA))_{|Y}$.
From this identity we deduce the desired inclusion 
\begin{align}
H^0(Y,\mathcal{O}_Y(kA))\subseteq S_{|Y,\bullet}\subseteq H^0(Y,\mathcal{O}_Y(kD)).
\end{align} 
\end{proof}

Since $\bsp(D)= \bigcap_{D=A+E} \text{Supp}(E)$, we recover the fact deduced in \cite{LM09}, that for a big divisor $D$ on $X$ and $Y\not\subseteq \bsp(D)$, the restricted linear series contains an ample series (satisfies condition (C)).

Now, we are able to prove our first slice formula for $t=0$ under the condition that $S_\bullet$ has full volume. It will follow as a corollary of the following.

\begin{thm} \label{thmsliceful}
Let $S_\bullet \subseteq T_\bullet$ be two finitely generated graded linear series such that $\vol(S_\bullet)=\vol(T_\bullet)>0$.
Suppose furthermore that $S_\bullet$ contains the ample series $D-E$.
 Then for all closed irreducible subvarieties $Y\not \subseteq \text{Supp}(E)$ we have
\begin{align}
\vol_{X|Y}(S_\bullet)=\vol_{X|Y}(T_\bullet).
\end{align}
\end{thm}
\begin{proof}
From the equality of volumes and the birationality of the maps $h_{S_\bullet}$ and $h_{T_\bullet}$ we can conclude, as in Theorem \ref{thmcharempty}, that the inclusion $R(S_\bullet)\subseteq R(T_\bullet)$ gives rise to a globally defined regular map:
\begin{align}
\proj(T_\bullet)\to \proj(S_\bullet).
\end{align}
Due to Lemma \ref{lemprojrest}, we arrive at the following commutative diagram:
\begin{align}
\xymatrix{ Y \ar@{-->}[rr]^{h_{T_{|Y,\bullet}}}\ar@{-->}[rrd]_{h_{S_{|Y,\bullet}}}& & \proj (T_{|Y,\bullet})\ar[d]^j \\
& & \proj (S_{|Y,\bullet}).}
\end{align}
By Lemma \ref{lembirrest}, the restricted series $S_{|Y,\bullet}$ and $T_{|Y,\bullet}$ contain an ample series. Hence, the maps $h_{T_{|Y,\bullet}}$ and $h_{S_{|Y,\bullet}}$ are both birational. Then we can conclude, as in Theorem \ref{thmcharempty}, that $\vol_{X|Y}(S_\bullet)=\vol_{X|Y}(T_\bullet)$.
\end{proof}
\begin{cor} \label{corslicefull}
Let $X$ be a normal projective variety.
Let $S_\bullet$ be a finitely generated graded linear series corresponding to a finitely generated divisor $D$ such that $\vol(S_\bullet)=\vol(D)$.
Suppose furthermore that $S_\bullet$ contains the ample series $D-E$.
 Then for all very admissible flags $Y_\bullet$ such that $Y_1$ does not contain the support of $E$ we have:
\begin{align}
\Delta_{Y_\bullet}(S_\bullet)_{\nu_1=0}=\Delta_{X|Y_1}(S_\bullet).
\end{align}
\end{cor}
\begin{proof}
From the above theorem we conclude that $\vol_{X|Y_1}(D)=\vol_{X|Y_1}(S_\bullet)$, which implies an equality $\Delta_{X|Y_1}(D)=\Delta_{X|Y_1}(S_\bullet)$. We have $\bsp(D)\subseteq E$  and therefore $Y_1\not\subseteq \bsp(D)$. Hence, we can use the slice formula \cite[Theorem 4.2.4]{LM09}  to conclude
\begin{align}
\Delta_{Y_\bullet}(S_\bullet)_{\nu_1=0}=\Delta_{Y_\bullet}(D)_{\nu_1=0}=\Delta_{X|Y_1}(D)=\Delta_{X|Y_1}(S_\bullet).
\end{align}

\end{proof}
Now, we want to get rid of the assumption that $S_\bullet$ has  to be finitely generated with full volume. However, the price we pay for this reduction  is the additional constraint that the point $Y_d$ of the flag be not contained in the stable base locus $\bs(S_\bullet)$.
\begin{thm} \label{thmslice}
Let $X$ be a normal projective variety.
Let $S_\bullet$ be a graded linear series containing the ample series $D-F$. Let $Y_\bullet$ be a very admissible flag such that  $Y_1$  is not contained in the support of $F$ and $Y_d\not \in \bs(S_\bullet)$. Then we have
\begin{align}
\Delta_{Y_\bullet}(S_\bullet)_{\nu_1=0}=\Delta_{X|Y_1}(S_\bullet).
\end{align}
\end{thm}
\begin{proof}
Let us first  treat the case where $S_\bullet$ is a finitely generated graded linear series generated by $S_1$.
Let $\pi \colon X^\prime \to X$ be the blow-up of the base ideal $\bsi_{S_1}$, $E$ the exceptional divisor and let $S^\prime_\bullet:= \pi^* S_\bullet -E$, as well as $D_E:= \pi^*D-E$.
Consider the decomposition $D=A+F$ into ample plus effective. There is a $k\gg 0$ such that  $\pi^*A-kE$ is very ample, we get an induced decomposition $\pi^*D=(\pi^*A-1/k\cdot E)+ (1/k\cdot E+\pi^*F)$.
Then it is easy to see that $S^\prime_\bullet$ contains the ample series $\pi^*A-1/k\cdot E=\pi^*D-(1/k\cdot E+\pi^*F)$.
Let $\tilde{Y_i}$ be the strict transform of $Y_i$ and $\tilde{Y}_{\bullet}$ be the corresponding flag on $\tilde{X}$.
Clearly, the strict transform $\tilde{Y_1}$ is not contained in the support of $ 1/k \cdot E+\pi^*F$ since $Y_1\not \subset\text{Supp}(F)$ and $\bs(S_\bullet)\subset \text{Supp}(F)$. 
Now we have 
\begin{align}
\Delta_{Y_\bullet}(S_\bullet)_{\nu_1=0}=\Delta_{\tilde{Y_\bullet}}(S^\prime_\bullet)_{\nu_1=0}=\Delta_{\tilde{X}|\tilde{Y_1}}(S^\prime_\bullet)
\end{align}
where the second  equality follows from Corollary \ref{corslicefull}.
To finish the first part of the proof we need to show  that $\Delta_{\tilde{X}|\tilde{Y_1}}(S^\prime_\bullet)=\Delta_{X|Y_1}(S_\bullet)$.
We have
\begin{align}
\Delta_{X|Y_1}(S_\bullet)\subseteq \Delta_{Y_\bullet}(S_\bullet)_{\nu_1=0}=\Delta_{\tilde{X}|\tilde{Y_1}}(S^\prime_\bullet).
\end{align}
The other inclusion follows from the fact that 
\begin{align}
\vol_{\tilde{X}|\tilde{Y_1}}(S^\prime_\bullet)= \vol_{\tilde{X}|\tilde{Y_1}}(\pi^*S_\bullet)=\vol_{X|Y_1}(S_\bullet)
\end{align}
where the first equality follows from the   bijection $S^\prime _{1|Y_1}\cong\pi^*S_1$ given by multiplication with the restriction of a defining section $s_E$ of $E$ to $Y^\prime_1$ and the last 
 equality follows from the property that $(\pi^*s)_{|\tilde{Y_1}}=(\pi_{|\tilde{Y_1}})^*(s_{|Y_1})$.
This proves the theorem for $S_\bullet$ being finitely generated.

Finally, we want to treat the case when $S_\bullet$ is not necessary finitely generated.
We will use Fujita approximation to reduce the statement to finitely generated graded linear series.
Define  the graded linear series $V_{\bullet,p}$ by \begin{align}
V_{k,p}:= \operatorname{Im} (\operatorname{Sym}^k(S_p)\to S_{kp}).
\end{align}
From \cite[Theorem 3.5]{LM09}, we deduce that for each $\varepsilon>0$, we can find $p_0$ such that for all $p\geq p_0$
\begin{align}
\vol (1/p\cdot\Delta_{Y_\bullet}(V_{\bullet,p}))\geq \vol (\Delta_{Y_\bullet}(S_\bullet))-\varepsilon.
\end{align}
Combining this with the easy fact that $1/p \cdot \Delta_{Y_\bullet}(V_{\bullet,p})\subseteq 1/p^\prime \cdot \Delta_{Y_\bullet}(V_{\bullet,p^\prime})$ for $p\leq p^\prime$ we have
\begin{align}
\Delta_{Y_\bullet}(S_\bullet)=\overline{\bigcup_{p\geq 0} 1/p\cdot \Delta_{Y_\bullet}(V_{\bullet,p})}. 
\end{align}
Analogously, we get 
\begin{align}
\Delta_{X|Y_1}(S_\bullet)=\overline{\bigcup_{p\geq 0} 1/p\cdot \Delta_{X|Y_1}(V_{\bullet,p})}. 
\end{align}

Combining these two properties leads to
\begin{align} 
\Delta_{Y_\bullet}(S_\bullet)_{\nu_1=0}&=
 \overline{\bigcup_{p\geq 0} ( 1/p\cdot \Delta_{Y_\bullet}(V_{\bullet,p})} \cap (\{0\}\times \mathbb{R}^{d-1}))=\\
&=\overline{ \bigcup_{p\geq 0} ( 1/p\cdot \Delta_{Y_\bullet}(V_{\bullet,p}) \cap  (\{0\}\times \mathbb{R}^{d-1})  })=\\
&=\overline{\bigcup_{p\geq 0} \left( \Delta_{X|Y_1}(V_{\bullet,p}) \right) }=\\ 
&= \Delta_{X|Y_1}(S_\bullet).  
\end{align}
Note that in the second equality we used the slice formula for finitely generated graded linear series.
This finishes the proof.
\end{proof}

We can apply the above theorem to the case of a restricted graded linear series. This enables us to get a generalization of   \cite[Theorem B]{jow}, which states that for a divisor $D$ and a curve $C$ which is constructed from intersecting $d-1$ very general very ample effective divisors $A_i$ on $X$. We have that the restricted volume of $\vol_{X|C}(D)$ is equal to the length of $\Delta_{Y_\bullet}(D)_{\nu_1=0,\dots, \nu_{d-1}=0}$ where $Y_i:=A_1\cap\dots \cap A_i$. 

\begin{cor}
Let $D$ be a divisor on $X$ and $Y_\bullet$ an admissible flag centered at $\{x\}\not \in \bs(D)$, such that $Y_i$ defines a Cartier divisor in $Y_{i+1}$.
Then for $Y_i\not\subseteq \bsp(D)$ we have
\begin{align}
\Delta_{X|Y_i}(D)=\Delta_{Y_\bullet}(D)_{\nu_1=0,\dots , \nu_{i}=0} .
\end{align}
\end{cor}
\qed

The last slice formula does not make any assumptions on the centered point $\{x\}$ of the flag, but  has more constraints on the divisorial component $Y_1$ of the chosen flag $Y_\bullet$.
\begin{thm}
Let $S_\bullet$ be a graded linear series that contains the ample series $A=D-E$. Let $Y_\bullet$ be a very admissible flag such that  $Y_1$ is not contained in the support of $E$ and is not fixed, i.e. there is  a natural number $k\in\mathbb{N}$ such that $h^0(X,\struc_{X}(k\cdot Y_1))> 1$. Then we have
\begin{align}
\Delta_{Y_\bullet}(S_\bullet)_{\nu_1=0}=\Delta_{X|Y_1}(S_\bullet).
\end{align}

\end{thm}
\begin{proof}

Let $\mathcal{S}_\bullet$ be the sheafification of $S_\bullet$. Consider the graded linear series $T_\bullet$ corresponding to the divisor $D+Y_1$ defined by the sheaves $\mathcal{S}_k\otimes_{\struc_{X}} \struc_{X}(k\cdot Y_1)$, i.e. defined by $T_k:=H^0(X,\mathcal{S}_k\otimes_{\struc_{X}} \struc_{X}(k\cdot Y_1))$. We want to show first   that the restricted linear  series $T_{|Y_1,\bullet}$ is not equal to zero. This follows from the fact that there is a non zero section $s\in H^0(X,\struc_X(k\cdot Y_1))$ which does not vanish at $Y_1$. Indeed, such a section exists. Let $s$ be a section in $H^0(X,\struc_X(k\cdot Y_1))$ which is not equal to a power of $s_{Y_1}$ up to a constant.  Let $a$ be the order of vanishing of $s$ along $Y_1$. By definition of $s_{Y_1}$, we have $a<k$ and $s/s_{Y_1}^{\otimes a} \in H^0(X,\struc_{X}((k-a)\cdot Y_1))$ does not vanish at $Y_1$.

Since $Y_1$ is not contained in the support of $E$, it is in particular not contained in the stable base locus of $S_\bullet$. Thus, we can pick a non-zero section $s^\prime \in S_k$ which does not vanish at $Y_1$.  Hence, the section $s^{\prime\prime}:=s^{\otimes k} \otimes s^\prime \in T_k$, does not vanish at $Y_1$.
This implies that $\nu_1(s^{\prime \prime})=0$.
Moreover, we can choose a section $s\in S_k$ such that $\nu_1(s)>0$. Then for the section $\tilde{s}:=s_E^{\otimes k}\otimes s\in T_k$ we have $\nu_1(\tilde{s})>1$. 
\begin{figure}[h] 
\centering
\begin{tikzpicture}
\draw[->] (-.5,0)--(5,0)node[below]{$e_1$};
\draw[->] (0,-.5)--(0,4) node[left]{$e_2$};

\draw[color=red] (1,2)-- (2,2.5);
\draw[color=red]  (2,2.5)-- (3,2.7);
\draw [color=red] (3,2.7)--(4,2.3);
\draw[style=dashed, color=red] (1,0.5)--(1,2);
\draw[color=red] (1,0.5)-- (2,0.7);
\draw[color=red] (2,0.7)--(3,1.2) ;
\draw[color=red] (3,1.2)-- (4,1.7);

\draw[color=red] (4,2.3)--(4,1.7) node[color=red,left=1.2cm] {$\Delta_{Y_\bullet}(S_\bullet)$};
\draw (0,0.4)--(1,0.5);
\draw(0,1.4)--(1,2)node[above=0.5cm]{$\Delta_{Y_\bullet}(T_\bullet)$}; 
\draw (1,0.2)--(1,-0.2) node[below] {$1$};
\end{tikzpicture}
\caption{Newton-Okounkov body $\Delta_{Y_\bullet}(T_\bullet)$}
\label{slicefig}
\end{figure}
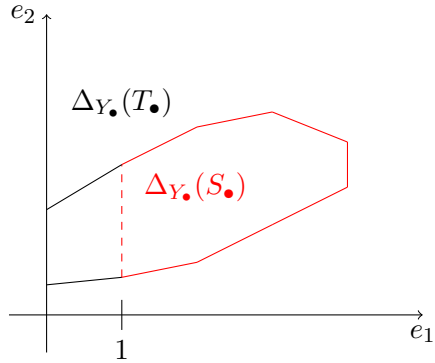
It follows from the above results on valuation vectors that the slice $\{1\}\times \mathbb{R}^{d-1}$ meets the interior of the Newton-Okounkov body $\Delta_{Y_\bullet}(T_\bullet)$.
By construction of $T_\bullet$, we have an isomorphism of sections $(T_\bullet-Y_1)_k\cong \tilde{S}_k$ where $\tilde{S}_\bullet$ is the sheafified graded linear series of $S_\bullet$.
With the help of Theorem \ref{thmslicepositive}, Corollary \ref{corslicepos} and Corollary \ref{sheafvol} we deduce:
\begin{align}
\Delta_{Y_\bullet}(S_\bullet)_{\nu_1=0}=& \Delta_{Y_\bullet}(\tilde{S}_\bullet)_{\nu_1=0}=\Delta_{Y_\bullet}(T_\bullet)_{\nu_1=1} \\
=&\Delta_{X|Y_1}(\tilde{S}_\bullet)=\Delta_{X|Y_1}(S_\bullet) .
\end{align}
Note that the last equality is due to Theorem \ref{thmsliceful}.
\end{proof}

\section{Generic Newton-Okounkov bodies } \label{secgen}
In this section we want to generalize the discussion in Chapter 5 of \cite{LM09}  to the case of birational  graded linear series $S_\bullet$. 
In order to define  Newton-Okounkov bodies, we have to fix  the variety $X$, the flag $Y_\bullet$ and a graded linear series $S_\bullet$, respectively a big divisor $D$. It was established in \cite[Theorem 5.1]{LM09}  that if we vary all the  different data $X,Y_\bullet$ and $D$ in a flat family, the resulting bodies all coincide  for a very general choice of these parameters. This allows to define  generic Newton-Okounkov bodies, called the infinitesimal Newton-Okounkov bodies, which do no longer depend on the choice of  a flag $Y_\bullet$.
The proof, which is presented in \cite{LM09}, relies heavily on the fact that $D$ induces a locally free sheaf $\mathcal{O}_X(D)$.
Hence, in order to generalize their results, we need to make use of the sheafification process considered in Section \ref{sectionsheaf}.  However, the resulting coherent sheaves $\mathcal{S}_k$ are not locally free, which also leads to technical difficulties to take into account. Finally, we will also get rid of the flatness hypothesis by using the theorem of generic flatness \cite[Corollary 10.84]{GW}.

\subsection{Family of Newton-Okounkov bodies} \label{famok}
Let us start by fixing the notation. Let $T$ be a (not necessarily projective) irreducible variety. This will be our parameter space. 
Let 
\begin{align}
\pi_T\colon X_T\to T
\end{align}
 be a family, such that for all $t\in T$ the fibers 
\begin{align} 
 X_t:= X_T\times_T k(t)
\end{align} 
  are projective varieties of dimension $d$. 
Let $S_{T,\bullet}$ be a graded linear series corresponding to a divisor $D_T$ on $X_T$ which is induced by a graded series of coherent sheaves $\mathcal{S}_{T,k}\subseteq \struc_{X_T}(k\cdot D_T)$. Furthermore, denote by $S_{t,\bullet}$  the graded linear series which is defined by taking the global sections of the pulled back sheaves $\mathcal{S}_{T,k|X_t}$. Additionally, we want to assume that $S_{t,\bullet}$ is a graded linear series corresponding to the divisor $D_t:=D_{T|X_t}$ as well as $\mathcal{S}_{t,k}$ are subsheaves of $\struc_{X_t}(k\cdot D_t)$.

 Let $\mathcal{Y}_\bullet$ be a partial flag of subvarieties
 \begin{align}
 X_T=\mathcal{Y}_0\supseteq \mathcal{Y}_1 \supseteq \cdots \supseteq \mathcal{Y}_d 
 \end{align}
 with the following additional properties. Denote the fibers of the flag 
$\mathcal{Y}_\bullet$ over $t\in T$  by 
 \begin{align}
 Y_{i,t}:= \mathcal{Y}_i\cap X_t.
 \end{align}
 The additional properties are:
\begin{enumerate}
\item Each $Y_{\bullet,t}$ is an admissible flag on $X_t$.
\item  The variety $\mathcal{Y}_{i+1}$ is a Cartier divisor in $\mathcal{Y}_i$.
\end{enumerate}

To summarize the above discussion, we give the following definition.
\begin{defi}\label{deffam}
Let $\pi_T\colon X_T\to T$, $S_{T,\bullet}$ and $\mathcal{Y}_{\bullet}$ be given such that all the above prescribed assumptions are fulfilled. Then we call $(X_T,S_{T,\bullet},\mathcal{Y}_\bullet)$ an \emph{admissible family of Newton-Okounkov bodies over  $T$}.
\end{defi}
Suppose we are given an admissible family of Newton-Okounkov bodies over $T$. Then for each $t\in T$, the following Newton-Okounkov body on $X_t$ is well defined:
\begin{align}
\Delta_{Y_{t,\bullet}}( S_{t,\bullet}).
\end{align}
The next lemma will be of help in the last section where we want to construct examples of admissible families of Newton-Okounkov bodies.
\begin{lem}\label{lempullgrad}
Let $p\colon X\to Y$ be a morphism of varieties. Let $S_\bullet$ be a graded linear series  corresponding to a divisor $D$ on $Y$ such that $S_1\neq \{0\}$. Suppose that $S_\bullet$ is given by taking global sections of a graded series of coherent sheaves $\mathcal{S}_k\subseteq \struc_Y(k\cdot D)$. Then the  pullback of the graded series of sheaves $p^*\mathcal{S}_k$ are coherent subsheaves of $\struc_X(k\cdot p^*D)$. Furthermore, by taking its global sections, it defines a graded linear series on $X$ (which we also denote by $p^*\mathcal{S}_\bullet$) if one of the following conditions are fulfilled:
\begin{itemize}
\item $p\colon X\to Y$ is flat, or 
\item $p\colon X\to Y$ is birational with $Y$ normal such that the image of the  exceptional locus is away from $\operatorname{Bs}(S_1)$.
\end{itemize}
Moreover, if $p$ is a morphism of projective varieties and satisfies the second condition, then $\operatorname{vol}(p^*\mathcal{S}_\bullet)=\operatorname{vol}(S_\bullet)$
\end{lem}
\begin{proof}
Let us first suppose that $p\colon X\to Y$ is flat. By the flatness of $p$, the sheaf $p^*\mathcal{S}_k$ is a coherent $\struc_X$-module which is contained in $p^*\struc_X(kD)$. 
Furthermore, for each non-negative  pair of integers $k,l$ the injection $\mathcal{S}_k\otimes \mathcal{S}_l\to \mathcal{S}_{k+l}$ pulls back to an injection
\begin{align}
p^*\mathcal{S}_k\otimes p^*\mathcal{S}_l\to p^*\mathcal{S}_{k+l}.
\end{align}
Therefore it is easy to see that $p^*\mathcal{S}_\bullet$ defines a graded linear series.

Now, let $p\colon X\to Y$ be birational, $Y$ normal and suppose there is an open subset $V\subseteq Y$ such that $\operatorname{Bs}(S_k)\subseteq \operatorname{Bs}(S_1)\subseteq V$ which induces an isomorphism $p_{|p^{-1}(V)}\colon p^{-1}(V)\to V$.
Now we claim that the induced canonical morphism 
\begin{align}
\kappa \colon p_*p^*\mathcal{S}_k\to\mathcal{S}_k
\end{align}
 is an isomorphism. We will prove this by showing that for each $y\in Y$ we find an open subset $U_y$ such that the induced morphism of sections $\kappa(U_y)$ is an isomorphism. For $y\in V$ choose  $U_y\subseteq V$. Then the induced morphism of sections is an isomorphism since $p_{|p^{-1}(U_y)}$ is an isomorphism. 
If $y\not \in V$ we can find an open neighborhood $U_y$ such that $U_y\subseteq Y\setminus \operatorname{Bs}(S_k)$.
But on $U_y$ the coherent sheaf $\mathcal{S}_{k|U_y}$ is invertible. Hence, we have an isomorphism 
\begin{align}
\mathcal{S}_{k|U_y}\cong \struc_Y(kD)_{|U_y}.
\end{align} 
However, for the locally free sheaf $\struc_Y(kD)$ the canonical morphism $p_*p^*\struc_X(kD)\cong \struc_X(kD)$ is an isomorphism.
This follows by using Zariski's Main theorem and the projection formula.
Hence, the canonical morphism $\kappa(U_y)$ is an isomorphism which completes the proof of the fact that $\kappa$ is an isomorphism.
We have the following commutative diagram of coherent sheaves on $Y$
\begin{align}
\xymatrix{p_*p^*\mathcal{S}_k\ar[r]^\cong \ar@{^{(}->}[d]& \mathcal{S}_k\ar@{^{(}->}[d]\\
			p_*p^*\struc_X(kD) \ar[r]^\cong & \struc_X(kD).} 
\end{align}
Taking global sections of the left vertical map induces an injection
\begin{align}
H^0(X,p^*\mathcal{S}_k)\to H^0(X,p^*\struc_X(kD)).
\end{align}
It remains to prove that the global sections define a graded algebra. We can proof exactly as before that we have an isomorphism of $\struc_Y$-modules $p_*p^*(\mathcal{S}_k\otimes \mathcal{S}_l)\cong \mathcal{S}_k\otimes \mathcal{S}_l$.
Again, we have a commutative diagram of coherent sheaves on $Y$ given by
\begin{align}
\xymatrix{p_*p^*(\mathcal{S}_l\otimes \mathcal{S}_k)\ar[r]^\cong \ar@{^{(}->}[d]& \mathcal{S}_l\otimes\mathcal{S}_k\ar@{^{(}->}[d]\\
p_*p^*\mathcal{S}_{l+k} \ar[r]^{\cong}& \mathcal{S}_{l+k}.}
\end{align}
Taking global sections of this diagram gives us an injection 
\begin{align}
H^0(X,p^*\mathcal{S}_l\otimes p^*\mathcal{S}_k)\to H^0(X,p^*\mathcal{S}_{l+k}).
\end{align}

 This implies that $p^*\mathcal{S}_\bullet$ defines a graded linear series.
Now we want to prove that $p^*\mathcal{S}_k\subseteq \struc_X(p^*D)$. This can again be checked by case distinction of open sets. Let $U$ be open such that $p_{|U}$ induces an isomorphism, then clearly 
\begin{align}
H^0(U,p^*\mathcal{S}_k)\cong H^0(p(U),\mathcal{S}_k)\subseteq H^0(p(U),\struc_{X}(k\cdot D))\cong H^0(U,\struc_{Y}(k\cdot p^*D)).
\end{align}
If $U\subseteq X\setminus p^{-1}(\operatorname{Bs}(S_k))$, consider the induced morphism $p\colon U \to Y\setminus \operatorname{Bs}(S_k):=W$.
Let $\bsi_{S_k}$ be the ideal sheaf of $S_k$, then $\mathcal{S}_{k|W}=(\struc_{X}(kD)\otimes \bsi_{S_k})_{|W}=\struc_{X}(kD)_{|W}$.
But this shows that $H^0(U,p^*\mathcal{S}_k)=H^0(U,p^*\struc_{Y}(kD))$.

The equality of volumes follows by  taking global sections of the canonical isomorphism $\kappa$ which gives an isomorphism
\begin{align}
H^0(X,p^*\mathcal{S}_k)\cong H^0(Y,\mathcal{S}_k).
\end{align}

\end{proof}
\begin{cor}\label{corsheafok}
Let $p\colon X\to Y$ be a morphism of projective varieties satisfying the properties of the second statement in the above lemma.
Let $Y_\bullet$ be an admissible flag on $X$ and $S_\bullet$ be birational graded linear series on $Y$ which is induced by the graded linear series of sheaves $\mathcal{S}_\bullet$. Then
\begin{align}
\Delta_{Y_\bullet}(p^*\mathcal{S}_\bullet)=\Delta_{Y_\bullet}(p^*S_\bullet).
\end{align}
\end{cor}
\begin{proof}
By Theorem \ref{thmsheafvol}, we can without loss of generality assume that $S_\bullet$ is induced by the graded series of sheaves $\mathcal{S}_\bullet$.
Since $\operatorname{vol}(p^*S_\bullet)=\operatorname{vol}(S_\bullet)=\operatorname{vol}(p^*\mathcal{S}_\bullet)$ it is enough to show one inclusion. But by construction $p^*\mathcal{S}_k$ contains all the elements $p^*s$ for $s\in S_k$. This shows that $\Delta_{Y_\bullet}(p^*S_\bullet)\subseteq \Delta_{Y_\bullet}(p^*\mathcal{S}_\bullet)$ and proves the claim.
\end{proof}
\begin{rem}
Note that by $p^*S_\bullet$ we denote the graded linear series, defined by considering the pullback of sections of $S_k$. However, by $p^*\mathcal{S_\bullet}$ we denote the graded linear series, given by taking the global sections of the pullback of coherent sheaves of $p^*\mathcal{S}_k$.
\end{rem}

\subsection{Partial sheafification of a graded linear series}
Let $X$ be a (not necessarily projective) variety, $D$ a  divisor and $\mathcal{S}$ a coherent subsheaf of $\struc_X(D)$. 
In this paragraph we want to generalize the discussion in \cite[Rem 1.4/1.5]{LM09}  to the sheaf $\mathcal{S}$.
Let $Y_\bullet$ be a partial flag of $X$ of length $r$ such that $Y_{i+1}$ is a Cartier divisor in $Y_{i}$.
Analogously as in the definition of Newton-Okounkov bodies, this partial flag defines  a valuation map
\begin{align}
\nu_{Y_\bullet}\colon H^0(X,\mathcal{S})\setminus \{0\} \to \mathbb{Z}^r.
\end{align}
If we fix a tuple $\sigma=(\sigma_1,\dots,\sigma_r)\in \mathbb{Z}^r$, we can define a subsheaf of $\mathcal{S}$ by setting for each open $U\subset X$ such that the induced flag $\mathcal{Y}_\bullet|U$ is of length $r^\prime\leq r$
\begin{align} \label{partialsheaf}
H^0(U,\mathcal{S}^{\geq (\sigma)}):=\{ s\in H^0(U,\mathcal{S}) \ | \ \nu_{Y_\bullet | U} (s) \geq (\sigma_1,\dots,\sigma_{r^\prime}) \}.
\end{align}
Here, the map $\nu_{Y_\bullet|U}$ is the valuation map corresponding to the restricted flag given by $Y_{i|U}:=Y_i \cap U$ on $U$. 

For the next theorem it is practical to make the following two abbreviations:
\begin{align}
\mathcal{S}(\sigma_1,\dots,\sigma_r):=&\mathcal{S}_{|Y_r}\otimes_{{Y_r}} \struc_X(-\sigma_1 Y_1)_{|Y_r}\otimes_{{Y_r}} \dots \otimes_{{Y_r}} \struc_{Y_{r-1}} (-\sigma_r Y_r)_{|Y_r} \\
\mathcal{S}(\sigma_1,\dots,\sigma_{r+1})_{|Y_r}:=&\mathcal{S}(\sigma_1,\dots,\sigma_r)\otimes_{Y_r}\struc_{Y_r}(-\sigma_{r+1} Y_{r+1}).
\end{align}
Note that these sheaves are both defined over $Y_r$. However, by a slight abuse of notation we will also consider them as sheaves over $X$ without writing them as pushforwards of the inclusion map $Y_r\hookrightarrow X$.

\begin{thm}
Let $\mathcal{S}$ be a coherent subsheaf of $\struc_X(D)$. Then for each partial flag $Y_\bullet$ of $X$ of length $r$ and $\sigma\in \mathbb{Z}^r$ there exists a coherent sheaf $\mathcal{S}^{\geq (\sigma)}$ such that \eqref{partialsheaf} holds and it induces a surjective morphism
\begin{align}
q_r\colon \mathcal{S}^{\geq (\sigma_1,\dots, \sigma_r)}\to \mathcal{S}(\sigma_1\dots, \sigma_r). 
\end{align}
\end{thm}
\begin{proof}
We will prove this using induction on $r$.
Let $r=1$. Then sections of $\mathcal{S}^{\geq (\sigma_1)}$ are those sections of $\mathcal{S}$ which vanish locally along $Y_1$ at least $\sigma_1$  times. These are given by the image of the following injection of coherent sheaves
\begin{align}
\mathcal{S}\otimes_{\struc_X} \struc_X(-\sigma_1 Y_1)\to \mathcal{S}.
\end{align} 
On an open set $U\subseteq X$ the above map of sections is defined by multiplication with a defining section $s_{Y_1 |U}$ of $Y_1$ to the power of $\sigma_1$.
Since $\mathcal{S}\otimes_{\struc_X} \struc_X(-\sigma_1 Y_1)$ is a coherent sheaf, this proves the claim for $r=1$.

Now let us suppose we have already defined $\mathcal{S}^{\geq (\sigma_1,\dots,\sigma_r)}$ as a coherent sheaf. Then we need to construct the sheaf $\mathcal{S}^{\geq (\sigma_1,\dots,\sigma_r,\sigma_{r+1})}$.
From the construction of the valuation $\nu_{Y_\bullet}$ we get a morphism of coherent sheaves
\begin{align}
q_r\colon \mathcal{S}^{\geq (\sigma_1,\dots, \sigma_r)}\to \mathcal{S}(\sigma_1\dots, \sigma_r). 
\end{align}
We claim that this morphism is surjective. We will prove this using the same induction on $r$.
For $r=1$, this map is just the restriction map to the closed subvariety $Y_1$
\begin{align}
q_1\colon \mathcal{S}^{\geq (\sigma_1)}\cong\mathcal{S}\otimes_X \struc_X(-\sigma_1 Y_1)\to \mathcal{S}_{|Y_1}\otimes_{Y_1} \struc_X(-\sigma_1 Y_1)_{|Y_1}
\end{align}
and hence surjective.
Now, let us additionally assume that $q_r$ is surjective.
We have the following natural inclusion
\begin{align}
\iota_r \colon \mathcal{S}(\sigma_1,\dots,\sigma_{r+1})_{|Y_r}\hookrightarrow \mathcal{S}(\sigma_1,\dots , \sigma_{r}) 
\end{align}
given by multiplication with a defining section of $Y_{r+1}$ in $Y_r$.
Now we define $\mathcal{S}^{\geq (\sigma_1,\dots, \sigma_{r+1})}$ as the preimage of $\mathcal{S}(\sigma_1,\dots, \sigma_{r+1})_{|Y_r}$ under the morphism $q_r$, yielding  the following diagram:
\begin{align}\label{partialsheafdiag}
\xymatrix{\mathcal{S}^{\geq (\sigma_1,\dots, \sigma_{r+1})} \ar@{->>}^-{p_r}[r]\ar@{^{(}->}[d] &   \mathcal{S}(\sigma_1, \dots , \sigma_{r+1})_{|Y_r}\ar@{^{(}->}[d]^{\iota_r} \\
 \mathcal{S}^{\geq (\sigma_1,\dots, \sigma_r)}\ar^-{q_{r}}@{->>}[r]& 
 \mathcal{S}(\sigma_1 ,\dots , \sigma_r).}
\end{align}
By the construction of the valuation, the sections of the coherent sheaf $\mathcal{S}^{\geq (\sigma_1,\dots, \sigma_{r+1})}$ are exactly the ones which satisfy equation \eqref{partialsheaf}. It remains to show that the morphism
\begin{align}
q_{r+1}\colon \mathcal{S}^{\geq (\sigma_1,\dots, \sigma_{r+1})} \to \mathcal{S}(\sigma_1, \dots, \sigma_{r+1})
\end{align}
is surjective.
But $q_{r+1}$ is just the composition of the surjection $p_r$ with the surjective restriction morphism. Hence, the surjectivity follows.
\end{proof}

\subsection{Generic Newton-Okounkov Body}

Let  ($X_T,S_{T\bullet},\mathcal{Y}_{\bullet})$ be an admissible family of Newton-Okounkov bodies over $T$. In this paragraph we want to prove that for a very general choice of  $t\in T$ the Newton-Okounkov bodies $\Delta_{Y_{\bullet,t}}(X_t, \mathcal{S}_{t,\bullet})$ all coincide.
The idea of the proof is to show that for a very general  choice $t\in T$ the dimension of the space of  global sections
\begin{align}
H^0(X_t, (\mathcal{S}_{k,t})^{\geq(\sigma)})
\end{align}
is independent from $t$.
The main issue of the proof is to show that we have an equality of coherent sheaves $(\mathcal{S}_{T,k}^{\geq (\sigma)})_{t}= (\mathcal{S}_{t,k})^{\geq (\sigma)}$. Once we have established this equality, we can use the theorem of generic flatness to deduce the constancy of the dimension.

We first need some helpful lemmata.
\begin{lem}
The commutative diagram constructed in \eqref{partialsheafdiag} gives rise to the following commutative diagram, where the rows are exact:

\begin{align} \label{sheafexdiag}
\xymatrixcolsep{1pc}\xymatrix{
0 \ar[r]& 	\mathcal{S}^{\geq (\sigma_1,\dots,\sigma_{r}+1)} \ar[r] \ar[d]^\cong &	\mathcal{S}^{\geq (\sigma_1,\dots, \sigma_{r+1})} \ar[r]^-{p_r}\ar@{^{(}->}[d] &   \mathcal{S}(\sigma_1,\dots , \sigma_{r+1})_{|Y_{r}}\ar@{^{(}->}[d]^{\iota_r} \ar[r] & 0 \\
0\ar[r] & \mathcal{S}^{\geq (\sigma_1,\dots,\sigma_{r}+1)} \ar[r] & \mathcal{S}^{\geq (\sigma_1,\dots, \sigma_r)}\ar[r]^-{q_{r}}& 
 \mathcal{S}(\sigma_1, \dots , \sigma_r)\ar[r]& 0.}
\end{align}

\end{lem}
\begin{proof} The only thing left to prove is the identity of the kernels of the horizontal maps $q_r$ and $p_r$.
We want to show that the induced map of sections  on the lower row is exact.
Let us first assume that $U\subseteq X$ is an open subset such that the induced flag $\mathcal{Y}_{\bullet|U}$ is of length $r^\prime<r$.
Then 
\begin{align}
H^0(U,\mathcal{S}^{\geq (\sigma_1,\dots,\sigma_{r}+1)})=H^0(U,\mathcal{S}^{\geq (\sigma_1,\dots,\sigma_{r^\prime})})= H^0(U,\mathcal{S}^{\geq (\sigma_1,\dots, \sigma_r)}).
\end{align}
Furthermore, $H^0(U,\mathcal{S}(\sigma_1,\dots,\sigma_r))=\{0\}$, since it is supported on $\mathcal{Y}_{r}$. This proves the exactness of sections on such $U$.
Now let $U$ be chosen such that $Y_{\bullet|U}$ is of maximal length $r$.
 We calculate the kernel of $q_r(U)$. Let $s\in H^0(U,\mathcal{S}^{\geq (\sigma_1,\dots,\sigma_{r}+1)})$. By definition of the valuation and the construction of the map $q_r$, the section $s$ will be sent to zero. However, if $\nu_{Y_{\bullet |U}}(s)=(\sigma_1,\dots, \sigma_r)$, then it is also clear that the image of $s\neq 0$ under $q_r$ does not vanish. This shows that $\ker (q_{r |U})=\mathcal{S}^{\geq (\sigma_1,\dots,\sigma_{r}+1)}$. The kernel of $p_r$ can be calculated using diagram \eqref{partialsheafdiag} as 
\begin{align}
\ker p_r= \ker q_r \cap \mathcal{S}^{\geq (\sigma_1,\dots,\sigma_{r},\sigma_{r+1})}= \mathcal{S}^{\geq (\sigma_1,\dots,\sigma_{r}+1)}.
\end{align}
\end{proof}

The next lemma seems to be common folklore knowledge. However, as a matter of a missing reference, we will prove this anyway.
\begin{lem}\label{lemflatbc}
Let $S$ be a noetherian scheme and $i\colon Z\to X$ be a closed immersion of noetherian $S$-schemes such that $Z$ is flat over $S$. Let $T$ be another $S$-scheme and $i_T\colon Z\times_S T \to X\times_S T$ be the closed immersion given by the following fiber diagram
\begin{align} \label{diagbc}
\xymatrix{Z\times_S T \ar[r]^{p_Z} \ar[d]^{i_T} & Z \ar[d]^i \\
X\times_S T \ar[r]^{p_X}& X.}
\end{align}
Then for each coherent $\struc_Z$-module $E$ we have a functorial isomorphism
\begin{align}
i_{T*}p_Z^*E\cong p_X^* i_*E.
\end{align}
\end{lem}
\begin{proof}
This question is local on $X$, $T$ and $S$. So let us set, without loss of generality, $X=\spec A$, $T=\spec B$ and $S=\spec R$. Then there is an ideal $I\subseteq A$ such that $Z=\spec (A/I)$. 
Consider the short exact sequence 
\begin{align}
0\to I \to A \to A/I\to 0.
\end{align}
Since $A/I$ is flat over $R$, we may tensor this sequence by $\otimes_R B$ and get
\begin{align}
0\to I\otimes_R B \to A \otimes_R B \to A/I \otimes_R B\to 0.
\end{align}

Hence, diagram \eqref{diagbc} gives rise to the following commutative diagram of rings
\begin{align}
\xymatrix{ A/I \ar[r]& (A\otimes_R B)/ (I\otimes_R B)  \\
A \ar[u] \ar[r]& A\otimes_R B. \ar[u]} 
\end{align}
Let $M$ be the $A/I$-module such that $\tilde{M}=E$. Then we have
\begin{align}
p_X^* i_*E= (  _A M \otimes_R A \otimes_R B)^{\sim}= (  _A M \otimes_A B)^{\sim}.      
\end{align}
Clearly, $I\cdot_AM=0$ and from this we deduce $(I\otimes_A B)({}_A M \otimes_A B )=0$. But this means that $p_X^* i_*E$ can be viewed as a sheaf over $(A\otimes_R B)/(I\otimes_R B)$. So there is a coherent $\struc_{Z\otimes_S T}$-module $L$ such that $i_{T*}L\cong p_X^* i_*E$. Taking $i_T^*$ of this isomorphism, gives us an isomorphism on $X\otimes_S T$
\begin{align}
i_T^*i_{T*} L \cong i_{T}^*p_X^*i_*E= p_Y^*i^*i_*E.
\end{align}
Since the canonical morphisms $i_T^*i_{T*} L\cong L $ and $i^*i_*E\cong E$ are isomorphisms, we conclude that $L\cong p_Y^*E$. Taking $i_{T*}$ of this isomorphism then gives the desired result.

\end{proof}
\begin{lem}\label{lemgen}Let $(X_T,S_{T,\bullet},\mathcal{Y}_\bullet)$ be an admissible family of Newton-Okounkov bodies and $(\sigma_1,\dots, \sigma_r)\in \mathbb{N}^{r}$. Let furthermore $\mathcal{Y}_i$ be flat over $T$ for all $i=1,\dots,d$. Consider the natural map
\begin{align}
\iota \colon \mathcal{S}_{T,k}(\sigma_1,\dots, \sigma_{r+1})_{|\mathcal{Y}_{r}}\hookrightarrow \mathcal{S}_{T,k}(\sigma_1,\dots, \sigma_r).
\end{align}
Viewing this as a map of coherent sheaves on $X_T$, for $t\in T$ the map $\iota$ 
 pulls back via the closed immersion $i_t\colon X_t\hookrightarrow X_T$ to the natural map
 \begin{align}
i_t^*\iota\colon \mathcal{S}_{t,k}(\sigma_1,\dots, \sigma_{r+1})_{|Y_{t,r}}\hookrightarrow \mathcal{S}_{t,k}(\sigma_1,\dots, \sigma_r)
\end{align}
viewed as a morphism of coherent sheaves on $X_t$.
 
\end{lem}
\begin{proof}
We consider the following fiber diagram 
\begin{align}
\xymatrix{Y_{r,t}=\mathcal{Y}_{r}\times_T k(t) \ar[r]\ar[d] & \mathcal{Y}_r \ar[d]\\
X_t=X_T\times_T k(t) \ar[r]& X_T}
\end{align}
By the previous Lemma \ref{lemflatbc}, the restriction of the coherent sheaf $\mathcal{S}_{T,\bullet}(\sigma_1,\dots, \sigma_r)$ on  $X_T$ to $X_t$ is the same as the restriction of $\mathcal{S}_{T,\bullet}(\sigma_1,\dots,\sigma_r)$ (viewed as a sheaf on $\mathcal{Y}_r$) to $Y_{r,t}$ and then viewing  it as a sheaf on $X_t$.
This means that
\begin{align}
\mathcal{S}_{T,\bullet}(\sigma_1,\dots,\sigma_r)_{|X_t}=\mathcal{S}_{T,\bullet}(\sigma_1,\dots,\sigma_r)_{|Y_{r,t}}.
\end{align}

For $i\leq r+1$ we have
\begin{align} 
 (\struc_{X_T}(-\sigma_i \mathcal{Y}_i)_{|\mathcal{Y}_r})_{|Y_{r,t}}= (\struc_{X_T}(-\sigma_i \mathcal{Y}_i)_{|X_t}))_{|Y_{r,t}}=
 \struc_{X_t}(-\sigma_i Y_{i,t})_{|Y_{r,t}}.
\end{align}
Since also $(\mathcal{S}_{T,k})_{| Y_{r,t}}=\mathcal{S}_{t,k|Y_{r,t}}$ we can follow that 
\begin{align}
 (\mathcal{S}_{T,k}(\sigma_1,\dots, \sigma_{r}))_{|X_t}\cong\mathcal{S}_{t,k}(\sigma_1,\dots, \sigma_r).
\end{align}
Similarly, we have 
\begin{align}
((\mathcal{S}_{T,k}(\sigma_1,\dots, \sigma_{r+1}))_{|\mathcal{Y}_{r}})_{|X_t} \cong\mathcal{S}_{t,k}(\sigma_1,\dots, \sigma_{r+1})_{|Y_{t,r}}
\end{align}
%which follows from the above discussion and the fact that
%\begin{align}
%\mathcal{S}_{T,k}(\sigma_1,\dots, \sigma_{r+1})_{|\mathcal{Y}_{r}}= \mathcal{S}_{T,k}(\sigma_1,\dots, \sigma_{r})\otimes_{\mathcal{Y}_r} \struc_{\mathcal{Y}_r}(-\sigma_{r+1}\mathcal{Y}_{r+1}).
%\end{align}
Finally, the restricted morphism $i_t^*\iota$ is given by multiplication with a defining section of $Y_{r}$ to the power of $\sigma_{r+1}$ and hence is injective.
\end{proof}

\begin{lem} \label{lemgen2}
Let ($X_T, S_{T,\bullet},\mathcal{Y}_\bullet)$ be an admissible family of Newton-Okounkov bodies over $T$. 
For a very general point $t\in T$ we have for every $k\in \mathbb{N}$  and $\sigma\in \mathbb{Z}^r$
\begin{align}
(\mathcal{S}_{T,k}^{\geq (\sigma)})_{t}= (\mathcal{S}_{t,k})^{\geq (\sigma)}.
\end{align}

\end{lem}
\begin{proof}
First let us fix the number $k$ and abbreviate $\mathcal{S}:= \mathcal{S}_k$.
We will prove the lemma using induction on $r$. 
Let $r=1$. Then we can identify $\mathcal{S}_T^{\geq (\sigma_1)}\cong \mathcal{S}_T \otimes_{X_T} \struc_{X_t}(-\sigma_1 \mathcal{Y}_1)$. Pulling this back to the fiber $X_t$ leads to
\begin{align}
(\mathcal{S}_T \otimes_{X_T} \struc_{X_T}(-\sigma_1 \mathcal{Y}_1))_{|X_t}&= (\mathcal{S}_T)_{|X_t}\otimes_{X_t} \struc_{X_t}(-\sigma_1\mathcal{Y}_1)_{|X_t}= \\
&= \mathcal{S}_t \otimes_{X_t} \struc_{X_t}(-\sigma_1 Y_{1,t})\cong \mathcal{S}_t^{\geq (\sigma_1)}.
\end{align}
This proves the lemma for $r=1$.

Now we prove the lemma for $r+1$ assuming that it holds for $r$.
Using the  theorem of generic flatness \cite[Corollary 10.85]{GW} we can find  open subsets $V_{\sigma,k}\subset T$ such that all coherent sheaves occurring in diagram \eqref{sheafexdiag} as well as the cokernels of the vertical morphisms are flat over $T$. Furthermore, we can also assume that all the $\mathcal{Y}_i$ for $i=1,\dots, r+1$ are flat over $V_{\sigma,k}$.
Then we let $t$ be in $\bigcap_{\sigma \in \mathbb{Z}^{r+1},k\in \mathbb{N}}{V_{\sigma,k}}$.
Due to the flatness, the induction hypothesis and Lemma \ref{lemgen}, we can pull back the right hand square of diagram  \eqref{sheafexdiag} and obtain the following square:
\begin{align}
\xymatrix{(\mathcal{S}_T^{\geq (\sigma_1,\dots, \sigma_{r+1})})_t \ar@{->>}[r]\ar@{^{(}->}[d] &   \mathcal{S}_t(\sigma_1,\dots, \sigma_{r+1})_{|Y_{r,t}}\ar@{^{(}->}[d]^{\iota_r} \\
 \mathcal{S}_t^{\geq (\sigma_1,\dots, \sigma_r)}\ar@{->>}[r]^{q_r}& 
 \mathcal{S}_t(\sigma_1,\dots , \sigma_r).}
\end{align}
The above diagram implies that we get an injection from $(\mathcal{S}_T^{\geq (\sigma_1,\dots, \sigma_{r+1})})_t$ to the inverse image of  $\mathcal{S}_t(\sigma_1,\dots, \sigma_{r+1})_{|Y_{r,t}}$ under the map $q_r$, which is by construction equal to  $\mathcal{S}_t^{\geq (\sigma_1,\dots ,\sigma_{r+1})}$.
We also have an isomorphism $\mathcal{S}_t(\sigma_1,\dots, \sigma_{r+1})\cong \left(\mathcal{S}_T(\sigma_1,\dots,\sigma_{r+1})\right)_t$ and therefore we get the following commutative diagram
\begin{align}
\xymatrixcolsep{1pc}
\xymatrix{ 0 \ar[r] &  (\mathcal{S}_T)^{\geq(\sigma_1,\dots,\sigma_{r+1}+1)}_t \ar[r] \ar@{^{(}->}[d]^{\iota_t}   &   (\mathcal{S}_T)^{\geq(\sigma_1,\dots,\sigma_{r+1})}_t \ar[r] \ar@{^{(}->}[d]^{\iota_t} &  \left(\mathcal{S}_T(\sigma_1,\dots,\sigma_{r+1})\right)_t \ar[d]^\cong  \ar[r]& 0 \\
 0 \ar[r] &  (\mathcal{S}_t)^{\geq(\sigma_1,\dots,\sigma_{r+1}+1)} \ar[r]    &   (\mathcal{S}_t)^{\geq(\sigma_1,\dots,\sigma_{r+1})} \ar[r]  &  \mathcal{S}_t(\sigma_1,\dots ,\sigma_{r+1})\ar[r] & 0. } 
\end{align}

We will use a second induction argument on $\sigma_{r+1}$. Let $\sigma_1,\dots,\sigma_{r}$ as well as $t\in T$ be fixed. Let $\sigma_{r+1}=0$. Then we have $\mathcal{S}_t^{\geq(\sigma_1,\dots,\sigma_r,0)}=\mathcal{S}_t^{\geq(\sigma_1,\dots,\sigma_r)}$ as well as $\mathcal{S}_T^{\geq(\sigma_1,\dots, \sigma_r,0)}=\mathcal{S}_T^{\geq(\sigma_1,\dots, \sigma_r)}$. Hence, for $\sigma_{r+1}=0$ the desired identity follows from the induction hypothesis on $r$.
Now let us assume, we know that the desired identity of sheaves is true for $\sigma_{r+1}$. Then we want to prove it is true for $\sigma_{r+1}+1$. However, this follows by using the above commutative diagram and the Five lemma. Indeed, by our induction hypothesis, the middle vertical morphism is an isomorphism. Hence, the left vertical morphism must be an isomorphism as well.
This proves the claim. 
\end{proof}

\begin{thm} \label{thmgenNO}
Let $(X_T,S_{T,\bullet},\mathcal{Y}_\bullet)$ be an admissible family of Newton-Okounkov bodies over $T$.
Then for a very general $t\in T$  the Newton-Okounkov bodies
\begin{align}
\Delta_{Y_{t,\bullet}}(S_{t,\bullet})
\end{align}
all coincide.
\end{thm}
\begin{proof}
 For a fixed $k\in \mathbb{N}$ and $\sigma\in \mathbb{N}^d$, there is an open subset $U_{\sigma,k}$ such that $\mathcal{S}^{\geq (\sigma)}_k$ is flat over $U_{\sigma,k}$ due to the theorem of generic flatness. Furthermore, by the the semicontinuity theorem, we can shrink $U_{\sigma,k}$ even more and have for all $t\in U_{\sigma,k}$ that the dimension of
\begin{align}
h^0(X,(\mathcal{S}_{T,k}^{\geq(\sigma)})_t)
\end{align}
is independent from $t$.
For a very general $t\in \bigcap_{k,\sigma} U_{\sigma,k}$ the constancy of the above dimension holds for every $k$ and $\sigma$.
Furthermore for $t\in \bigcap_{\sigma,k} U_{\sigma,k} \cap \bigcap_{\sigma,k} V_{\sigma,k}$, we have that
$h^0(X,(\mathcal{S}_{T,k}^{\geq(\sigma)})_t)=h^0(X,\mathcal{S}_{t,k}^{\geq(\sigma)})$ are indepent from $t$.
But for a fixed $t\in T$ the dimensions of all the $H^0(X,\mathcal{S}_{t,k}^{\geq (\sigma)})$ completely determine the valuation points of $\Delta_{Y_{t,\bullet}}(S_{t,\bullet})$ and thus also the body $\Delta_{Y_{t,\bullet}}(S_{t,\bullet})$
 From this observation it follows that for a very general $t$ all Newton-Okounkov bodies coincide.
\end{proof}

\subsection{Examples of generic Newton-Okounkov bodies}

In this paragraph we want to construct some admissible families of Newton-Okounkov bodies over $T$, in order to illustrate how to make use of Theorem \ref{thmgenNO} to get generic Newton-Okounkov bodies. We will give three construction how to realize this. 
The first two examples are families where we just vary the flag $Y_\bullet$. In the third construction we will also vary the varieties $X_t$ and the graded linear series $S_{t,\bullet}$ in a family. For the sake of simplicity, we assume in this paragraph that alle the varieties $X$ occuring are smooth.

\subsubsection{Variation of the flag}

Let $S_\bullet$ be a birational graded linear series in $X$ corresponding to a divisor $D$ and let $T$ be an irreducible (not necessarily projective) variety.
Since $S_\bullet$ is birational, we can consider the sheafification $\mathcal{S}_\bullet$ of $S_\bullet$. By Corollary \ref{sheafvol}, we may replace $S_\bullet$ by $\tilde{S}_\bullet$ and may therefore assume that $S_\bullet$ is induced by the family of sheaves $\mathcal{S}_\bullet$.
Consider the variety $X_T:=X\times_{\mathbb{C}}T$ and the projection
\begin{align}
\pi_{X}\colon X_T\to X.
\end{align}
We define the family of sheaves $\mathcal{S}_{T,\bullet}$ as the pullback of $\mathcal{S}_\bullet$ under the projection $\pi_X$.
Since $\pi_X$ is flat, we can use Lemma \ref{lempullgrad} to see that  the family of sheaves $\mathcal{S}_{T,\bullet}$ defines a graded linear series on $X_T$ and that the sheaves $\mathcal{S}_{T,k}$ are subsheaves of $\struc_{X_T}(k\pi_X^*D)$.
 In order to get a family of Newton-Okounkov bodies, it remains to choose an admissible flag $\mathcal{Y}_\bullet$ of $X_T$.
So let us suppose we have fixed such a flag $\mathcal{Y}_\bullet$. For a point $t\in T$, the fiber over $t$ of $X_T$ induces an isomorphism
$X_t\cong X$. It is not hard to see that the following composition of morphism is an isomorphism
\begin{align}
X\cong X_t \hookrightarrow X\times_{\mathbb{C}}T\xrightarrow{\pi_X} X
\end{align}
where the first map is the natural inclusion of the fiber.
This shows that $\mathcal{S}_{t,\bullet}\cong \mathcal{S}_\bullet$ and hence  $(X_T,\mathcal{S}_{T,\bullet},\mathcal{Y}_\bullet)$ defines an admissible family of Newton-Okounkov bodies.

\begin{ex}[Variation of the point $Y_d$]
Let us suppose we have a birational graded linear series $S_\bullet$ on a smooth projective variety $X$ and a partial admissible flag of smooth subvarieties
\begin{align}
Y_1\supseteq \dots \supseteq Y_{d-1}
\end{align}
fixed.
Then set $T:=Y_{d-1}$ and consider  the variety $ X_T:=X\times_{\mathbb{C}} T$ and the partial flag $\mathcal{Y}_\bullet$ defined by $\mathcal{Y}_i:= Y_i\times_\mathbb{C}Y_{d-1}$ for $i=1,\dots, d-1$ and $\mathcal{Y}_d:=Y_{d-1}$ which is embedded in $\mathcal{Y}_{d-1}=Y_{d-1}\times_{\mathbb{C}} Y_{d-1}$ via the diagonal embedding.
Then for each $x\in T=Y_{d-1}$, the flag $Y_{x,\bullet}$ is just the partial flag $Y_1,\dots, Y_{d-1}$ with the additional component $Y_{x,d}=\{x\}$. Hence, Theorem \ref{thmgenNO} implies that for a very general point $x$ in $Y_{d-1}$
the Newton-Okounkov bodies $\Delta_{\{Y_1\supseteq \dots \supseteq \{x\}\}}(S_\bullet)$ all coincide.

We will now show that for  the special case of a surface $X$ and a finitely generated birational graded linear series this result can be established in a more direct way and also holds for  a general choice of points of the flag.
Let $X$ be a smooth surface, $D$ a big divisor on $X$ and $C$ a smooth curve. Then for each $x\in C$ we obtain an admissible flag $X\supseteq C\supseteq \{x\}$. We have the following description of the Newton-Okounkov body on surfaces (see also \ref{oksurf}):
\begin{align} 
\Delta_{C\supset\{x\} }(D)= \{ (t,y)\in \mathbb{R}^2: a\leq t\leq \mu, \quad \text{and } \alpha(t)\leq y \leq \beta(t) \}.
\end{align}
Without loss of generality we may replace $D$ by $D-aC$ and assume that $C$ is not contained in the support of the negative part of $D$. Then $\alpha(t)=\operatorname{ord}_x (N_t|C)$ and $\beta(t)=\operatorname{ord}_x(N_t|C)+(C\cdot P_t)$, where $D_t:=D-tC= P_t+N_t$ is the Zariski decomposition.
However, it is an easy consequence of \cite[Proposition 2.1]{KLM12}  that, the support of  all $N_t$ is contained in a finite union of closed subvarieties. Hence, we can choose a general point $x\in C$, such that $x\not \in \operatorname{supp} (N_t)$ for each $t\in [0,\mu]$. Then $\operatorname{ord}_x(N_t|C)=0$ and $\alpha(t)$ as well as $\beta(t)$ do not depend on the general point $x\in C$. This shows that for a general choice of $x\in C$, the Newton-Okounkov body $\Delta_{\{C\supseteq \{x\}\}}(D)$ is independent from $x$.

Now let $S_\bullet$ be a finitely generated birational graded linear series on $X$. Without loss of generality we may assume that it is finitely generated in $S_1$.
Let $\pi\colon X^\prime \to X$ be the blow-up of $X$ along $\bsi_{S_1}$ and let $\tilde{C}$ be the strict transform of $C$. Without loss of generality, we may assume that $X^\prime$ and $\tilde{C}$ are smooth. If this does not hold we can pass to a resolution of singularities, without changing the Newton-Okounkov body. 
By Theorem \ref{thmgradvollb}, we have for all $x\in \tilde{C}\setminus \bs(\pi^*S_\bullet)$:
\begin{align}
\Delta_{\{\tilde{C}\supseteq \{x\}\}}(\pi^*D-E)=\Delta_{\{C\supseteq\{x\}\}}(S_\bullet).
\end{align}
But the above discussion shows that the left hand side does not depend on $\tilde{x}$ for a general choice. Hence, also the right hand side does not.  

\end{ex}

\begin{ex}[Flags of complete intersection of very ample divisors]
In this example we want to consider flags which are defined by complete intersections corresponding to global sections of a fixed very ample divisor. We will see that the family of such flags induces an admissible flag. Thus we can define a generic Newton-Okounkov body corresponding to a birational graded linear series $S_\bullet$ on $X$, which just depends on the choice of a very ample divisor $A$. 

Consider the variety $S^\prime:= \mathbb{P}(H^0(X,\struc_X(A)))^{d-1}$. By Bertini's Theorem, there is an open subvariety $S\subset S^\prime$ such that for all $([s_1],\dots,[s_{d-1}])\in S^\prime$, the variety cut out by the $s_1,\dots, s_i$ 
\begin{align}
Y_{i}=\{x\in X \ | \ s_1(x)=\dots=s_i(x)=0\}
\end{align}
 for $i=1,\dots, d-1$ are smooth of codimension $i$ in $X$.
Consider the variety
\begin{align}
T:=\{(x,s_1,\dots,s_{d-1})\in X\times_{\mathbb{C}} S \ | \ s_1(x)=\dots = s_{d-1}(x)=0\}. 
\end{align}
as our parameter space,  as well as $X_T=X\times_{\mathbb{C}} T$ as our total space.
Note that $T$ is irreducible since, it surjects into $S$ which is irreducible and the fibers $T_s$ are irreducible curves for each $s\in S$.

Then we can define the partial flag $\mathcal{Y}_\bullet$ by setting
\begin{align}
\mathcal{Y}_i:=\{(x,y,[s_1],\dots, [s_{d-1}])\in X_T\subseteq X\times_\mathbb{C} X\times_\mathbb{C} S \ | \ s_1(x)=\dots = s_{i}(x)=0\}.
\end{align}
From the construction  it  follows that for each $t=(y,[s_1],\dots,[s_{d-1}])\in T$, the induced flag  $Y_{t,\bullet}$ consists of  the smooth varieties $Y_{t,i}$ defined above for $i=1,\dots d-1$ and $Y_{t,d}=\{y\}$.
Now, we want to show that the $\mathcal{Y}_i$ are Cartier divisors in $\mathcal{Y}_{i-1}$. We may without loss of generality replace the variety $T$ by an open subset $U\subseteq T$. Then we can assume that $T$ is smooth and all the $\mathcal{Y}_i$ are flat over $T$. Since all the fiber $Y_{i,t}$ for $t\in T$ are smooth, we can use 
\cite[Proposition 14.57]{GW} to deduce that $\mathcal{Y}_i$ is smooth as well. Hence, all the $\mathcal{Y}_i$ can be considered as Cartier divisors in $\mathcal{Y}_{i-1}$.

We have shown that $(X_T,S_{T,\bullet},\mathcal{Y}_\bullet)$ is an admissible family of Newton-Okounkov bodies and can therefore use Theorem \ref{thmgenNO} to get a generic Newton-Okounkov body $\Delta_{A}(S_\bullet)$ corresponding to the very ample line bundle $A$ and the birational graded linear series $S_\bullet$.
\end{ex}
\subsubsection{Infinitesimal Newton-Okounkov body}
Finally we do not just want to vary the flag $Y_\bullet$ but also the variety $X$ by considering blow-ups at various points on a variety. So let us fix a birational graded linear series $S_\bullet$ on a smooth variety $X$. Then if we choose a point $x\in X$, we denote by $X_x$ the blow-up of $X$ at $x$. Let $E=\mathbb{P}(T_xX)$ be the exceptional divisor and $\pi \colon X_x \to X$ the corresponding blow-up morphism. Then for each choice of flags of vector spaces
\begin{align}
T_xX= V_0\supseteq V_1,\supseteq \dots \supseteq V_{d-1}\supseteq \{0\}
\end{align}
we get an induced linear flag  $\mathbb{P}(V_{\bullet})$ defined by
\begin{align}
\mathbb{P}(T_x X)=E= \mathbb{P}(V_0)\supseteq \mathbb{P}(V_1),\supseteq \dots \supseteq \mathbb{P}(V_{d-1})\supseteq \{pt\}
\end{align}
 on $X_x$ starting with $E$. Hence, we  can define \begin{align}
 \Delta_{F(x,V_\bullet)}:=\Delta_{V_\bullet}(\pi^*S_\bullet),
\end{align} which we call an infinitesimal Newton-Okounkov body (see also \cite[Section 5.2]{LM09} ).  We want to see that this construction varies in an admissible family of Newton-Okounkov bodies.
The following lemma is a first step. 
\begin{lem}
Let $X$ be a smooth projective variety.
There is a smooth projective variety $B$ and a projection $p\colon B\to X$ such that for each $x\in X$ the fiber $B_x$ is isomorphic to $X_x$ which is the blow-up of $X$ at the point $x$.
\end{lem}
\begin{proof}
Consider the diagonal closed embedding  $X\hookrightarrow X\times_\mathbb{C} X$. Let $\pi \colon B:=Bl_{X}(X\times_\mathbb{C}X)\to X\times_{\mathbb{C}}X$ be the blow-up of the closed variety $X$ inside $X\times_\mathbb{C} X$ with respect to the above embedding.
We consider $B$ as a family over $X$ by 
\begin{align}
B\xrightarrow{\pi} X\times_\mathbb{C}X\xrightarrow{p_2} X
\end{align}
where $p_2$ is the projection on the second factor.
Let $x\in X$ be a closed point. Then we  make the following abbreviations:
\begin{align}
\{x\}:=\spec k(x) \quad X\times_{\mathbb{C}}\{x\}:= X\times_\mathbb{C}X\times_{X}\spec k(x).
\end{align}
Now we have the following commutative diagram:
\begin{align}
\xymatrix{ 
&B\ar[r]^{\pi} &  X\times_\mathbb{C}X \ar[r]^-{p_2}& X \\
 &B\times_{X\times  X}( X\times_\mathbb{C} \{x\})\ar@/_/@{.>}[ld]_{q^\prime}\ar@{^{(}->}[u]^{p_1}\ar[r]^-{p_2}& X\times_\mathbb{C}\{x\} \ar@{^{(}->}[u]\ar[r]^-{p_2}& \{x\}\ar@{^{(}->}[u]\\
 Bl_{(x,x)}(X\times_\mathbb{C} \{x\})  \ar[ru]_{q}\ar@{^{(}->}@/^/[ruu]\ar@/_/[rru] &&& .} 
\end{align}
Note that the two rightmost boxes are Cartesian and the map $q$ is given by the universal property of fiber products.
We want to prove that the morphism $q$ is an isomorphism. We will do this by constructing the inverse map, using the universal property of the blow-up map
\begin{align}
Bl_{(x,x)}(X\times_\mathbb{C}\{x\})\to X\times_\mathbb{C}\{x\} .
\end{align}
This is constructed by showing that the inverse image of $(x,x)\in X\times_\mathbb{C} \{x\}$ under the map
\begin{align}
p_2\colon  B\times_{X\times  X}( X\times_\mathbb{C} \{x\}) \to X\times_\mathbb{C}\{x\}
\end{align}
is an effective Cartier divisor.  But this inverse image is the same as the the inverse image of the exceptional divisor $\pi^{-1}(X)\subset B$ under the map $p_1$. However, one can easily see that that the image of $p_1$ is not contained in the exceptional divisor $\pi^{-1}(X)$ and hence, we can pull back the Cartier divisor by just pulling back the local equation.
Consequently, we have defined a map $q^\prime$ which fits into the above commutative diagram. Now, by the universal property of the blow-up and the universal property of the fiber product in the middle box, we deduce that $q^\prime$ is an inverse map of $q$.

Thus, the fiber $B_x$ which is just $B\times_{X\times  X}( X\times_\mathbb{C} \{x\})$ is isomorphic to $Bl_{(x,x)}(X\times_\mathbb{C}\{x\})$
which can be interpreted as the blow-up $Bl_x(X)$ of $X$ in the point $\{x\}$.

\end{proof}
Let us denote by $N:=N_{X/X\times X}$ the normal bundle of the diagonal embedding, viewed as a vector bundle over $X$. The projectivization $\mathbb{P}(N)$ is the exceptional divisor of the blow-up $B$. Its fibers $\mathbb{P}(N)_x$ are isomorphic to the exceptional divisors $E_x$ of $X_x$.
Let $T:=Fl(N)\xrightarrow{p} X$ be the flag bundle of $N$ over $X$. By the splitting principle, there exists a filtration of vector bundles of $p^*N=N\times_X T$:
\begin{align}
p^*N=N\times_X T\supset \mathcal{V}_1\supset \dots \supset \mathcal{V}_{d-2} \supset \{0\}.
\end{align}
We can also consider the projectivized filtration:
\begin{align}
\mathbb{P}(N)\times_X T\supset \mathbb{P}(\mathcal{V}_1)\supset \dots \supset \mathbb{P}(\mathcal{V}_{d-2}).
\end{align} 
Now we define $X_T:=B\times_X T$, as well as $\mathcal{Y}_1:=\mathbb{P}(N)\times_X T$ which is a subset of codimension one in $X_T$ since $T\xrightarrow{p} X$ is flat. Furthermore, we define $\mathcal{Y}_i:= \mathbb{P}(\mathcal{V}_{i+1})$.
The so defined flag $\mathcal{Y}_\bullet$ is admissible and has the desired properties on the fibers over $t\in T$ since $\mathbb{P}(N)\times_X k(t)$ is isomorphic to the exceptional divisor of the blow-up of $X$ in $p(t)$.

In order to define a graded linear series induced by a graded series of sheaves, we need to shrink $T$ a bit more.
Let $S_\bullet$ be a graded linear series on $X$. Let $Z:=\bs(S_\bullet)$ be the corresponding base locus and define $U=X\setminus Z$. Then we consider the base change $T^\prime:=T\times_X U\xrightarrow{p^\prime} U$ which is an open subset of $T$ and define $X_{T^\prime}:=X_{T}\times_X T^\prime=B\times_X T\times_X U$.
Now consider the following composition of morphisms of varieties
\begin{align}
X_{T^\prime, t^\prime}\hookrightarrow X_{T^\prime}\to (B\times_X U)\to (X\times X)\times_X U\xrightarrow{p_1} X.
\end{align}
Let $\mathcal{S}_\bullet$ be the sheafification of $S_\bullet$. We can now use Lemma \ref{lempullgrad} to deduce that the pullback of $\mathcal{S}_\bullet$ to $X_{T^\prime}$, which we define as $\mathcal{S}_{T^\prime,\bullet}$, defines a graded linear series since it factors as the pullback of a flat morphism composed with a birational morphism which has the prescribed property of Lemma \ref{lempullgrad}, again composed with a flat morphism.
Furthermore, the composed map $X_{T^\prime,t^\prime}\to X$ is just the blow-up morphism of $X$ in $p^\prime(t)\in U$. Again we can use Lemma \ref{lempullgrad} to deduce that $\mathcal{S}_{t,\bullet}$ defines a graded linear series on $X_{T^\prime,t^\prime}$ and $\mathcal{S}_{t,k}\subseteq \struc_{X_x}(D_{|X_x})$. Corollary \ref{corsheafok}, then says that the Newton-Okounkov body of $\mathcal{S}_{t,\bullet}$ is the same as for the graded linear series $\pi^*S_{\bullet}$ which is given by pulling back the global sections.
We can also replace the flag $\mathcal{Y_\bullet}$ with the flat base change $\mathcal{Y}^\prime_\bullet:=(\mathcal{Y}\times_X U)_\bullet$.
Then  it follows from our discussion that $(X_{T^\prime},\mathcal{Y}^\prime_\bullet, \mathcal{S}_{T^\prime,\bullet})$ is an admissible family of Newton-Okounkov bodies over $T^\prime$.
Let us summarize what we have shown.
\begin{thm}
Let $X$ be a smooth variety and $S_\bullet$ be a birational graded linear series. Then for a very general choice of points $p\in X$ and a linear flag $V_\bullet$ starting with  $E_x:=\mathbb{P}(T_x(X))\cong \mathbb{P}^{d-1}$, the corresponding Newton-Okounkov bodies $\Delta_{F(x,V_\bullet)}(S_\bullet)$
all coincide.
\end{thm}
\qed


\begin{thebibliography}{999999}



\bibitem[A13] {A13}
Anderson, D.,
{\it Okounkov bodies and toric degenerations}
{ Math. Ann.} {\bf 356} (2013), no. 3, 1183-1202


\bibitem[AKL12]{AKL12}
Anderson, D., K\"uronya, A., Lozovanu, V., 
{\it Okounkov bodies of finitely generated divisors}, 
International Mathematics Research Notices
2013; doi: 10.1093/imrn/rns286


%\bibitem[BCL]{BCL}
%Boucksom, S., Cacciola, S., Lopez, A. F.,
%{\it Augmented base loci and restricted volumes on normal varieties }
%{Math. Z. } {\bf 278 } (2014), n. 3-4, 979-985


\bibitem[BG09]{BG09}
 Bruns, W., Gubeladze, J.,{\it Polytopes, rings, and K-theory}, Springer, Dordrecht, (2009).


\bibitem[ELMNP09]{ELMNP09}
 Ein, L.,Lazarsfeld, R., Musta{\c{t}}{\u{a}}, M., Nakamaye,M.,Popa, M.,
{\it Restricted Volumes and Base Loci of Linear Series }
{American Journal of Mathematics}
{\bf 131}(2009), No. 3 , 607-651

%\bibitem[H77]{H77}
%{Hartshorne, R.}, 
%{\it Algebraic Geometry}, Springer, 1977
\bibitem[GW]{GW}
{G\"ortz, U., Wedhorn, T., }
``Algebraic Geometry: Part I: Schemes. With Examples and Exercises (Advanced Lectures in Mathematics)", Vieweg+Teubner Verlag, 2010

\bibitem[H13]{H13}
{Hisamoto,T.},
{\it On the volume of graded linear series and Monge Ampère mass},{ Math. Z.} {\bf 275} (2013), 233--243

%\bibitem[HK15]{HK15}
%Harada, M., Kaveh, K.,
%{\it Integrable systems, toric degenerations and Okounkov bodies }
%{Inventiones mathematicae} {\bf 202} (2015) Issue 3, 927-985




\bibitem[J10]{jow}
{Jow, Shin-Yao},
{\it Okounkov bodies and restricted volumes along very general
              curves},
Adv. Math. {\bf 223}, (2010), 4, 1356--1371

\bibitem[KK12]{KK12}
{Kaveh, K., Khovanskii, A.G.},
{\it Newton convex bodies, semigroups of integral points, graded algebras 
and intersection theory}, \newblock  Ann. of Math. \textbf{176} (2012), 925--978.



\bibitem[KL14]{KL14}
{K\"uronya, A., Lozovanu, V.},
{\it Local Positivity of linear Series on Surfaces},
accepted in Algebra \& Number Theory, arXiv:1411.6205



\bibitem[KL15]{KL15}
{K\"uronya, A., Lozovanu, V.},
{\it Positivity of line bundles and Newton-Okounkov bodies},
 Documenta Math., \textbf{22} (2017), pp. 1285--1302

\bibitem[KL17]{KurLoz15}
A.~K\"uronya and V.~Lozovanu.
\newblock \emph{Infinitesimal Newton-Okounkov bodies and jet separation}. 
\newblock Duke Math. J. \textbf{166} (2017), 1349--1376.



\bibitem[KLM12] {KLM12}
K{\"u}ronya, A.,  Lozovanu, V.,Maclean, C.,
{\it Convex bodies appearing as Okounkov bodies of divisors }
{Advances in Mathematics} {\bf 229} (2012), 2622–-2639


\bibitem[KLM13] {KLM13}
K{\"u}ronya, A.,  Lozovanu, V.,Maclean, C.,
{ \it Volume functions of linear series},
{Mathematische Annalen}{\bf 356 } (2013), 635--652


\bibitem[KMS12]{KMS12}
K{\"u}ronya, A., Maclean, C., Szemberg, T.,
{\it Functions on Okounkov bodies coming from geometric valuations (with an appendix by S\'ebastien Boucksom) }
preprint, arXiv:1210.3523




\bibitem[L04]{laz}
{Lazarsfeld, R.,}
``Positivity in Algebraic Geometry I", Springer, 2004

\bibitem[LM09]{LM09}
{Lazarsfeld, R., Musta{\c{t}}{\u{a}}, M.},
{\it Convex bodies associated to linear series},
{Ann. Sci. \'Ec. Norm. Sup\'er.} {\bf 42}
(2009), 783--835

\bibitem[O96]{Ok96}
{Okounkov, A.},
{\it Brunn-{M}inkowski inequality for multiplicities},
{Invent. Math.} {\bf 125} (1996), 405--411


\bibitem[SS17]{SS17}
Schmitz, D., Sepp\"anen, H., {\it Global Okounkov bodies for Bott-Samelson varieties}, 
Journal of Algebra {\bf 490} (2017), 518--554











\end{thebibliography}
\end{document}